\documentclass[a4paper,10pt]{article}

\usepackage{graphicx}
\usepackage{geometry}
\usepackage{amsthm}
\usepackage{amssymb}
\usepackage{amsmath}
\usepackage{hyperref}
\usepackage{xcolor}
\usepackage{enumerate}
\usepackage{listings}
\usepackage{complexity}
\usepackage{blkarray}
\usepackage{lmodern}
\usepackage[english]{babel}
\usepackage{accents}
\usepackage[font=small,labelfont=bf]{caption}
\usepackage[font=small,labelfont=normalfont,labelformat=simple]{subcaption}
\graphicspath{{figs/}}

\newtheorem{theorem}{Theorem}
\newtheorem{definition}[theorem]{Definition}
\newtheorem{lemma}{Lemma}[section]
\newtheorem{corollary}[theorem]{Corollary}

\newtheorem{observation}[theorem]{Observation}
\newtheorem{proposition}[theorem]{Proposition}
\newtheorem{conjecture}[theorem]{Conjecture}
\newtheorem{problem}[theorem]{Problem}

\newcommand{\revvec}[1]{\accentset{\leftarrow}{#1}}

\newcommand{\bivec}[1]{\accentset{\leftrightarrow}{#1}}

\author
{
Lior Gishboliner \thanks{School of Mathematical Sciences, Sackler Faculty of Exact Sciences, Tel-Aviv University, Israel, email: \texttt{liorgis1@mail.tau.ac.il}.  Research supported by ERC Starting Grant 633509.}
\and
Raphael Steiner \thanks{Institute of Mathematics, Technische Universit\"at Berlin, Germany, email: \texttt{steiner@math.tu-berlin.de}.
Funded by DFG-GRK 2434 Facets of Complexity.}
\and
Tibor Szab\'{o} \thanks{Institute of Mathematics, Freie Universit\"{a}t Berlin, Germany, email: \texttt{szabo@math.fu-berlin.de}. 
Research supported in part by GIF grant No. G-1347-304.6/2016 and by the Deutsche Forschungsgemeinschaft (DFG, German Research Foundation) under Germany's Excellence Strategy - The Berlin Mathematics Research Center MATH+ (EXC-2046/1, project ID: 390685689).}
}

\date{\today}

\title{Dichromatic number and forced subdivisions}

\begin{document}
\maketitle


\begin{abstract}
	We investigate bounds on the dichromatic number of digraphs which avoid a fixed digraph as a topological minor. For a digraph $F$, denote by $\text{mader}_{\vec{\chi}}(F)$ the smallest integer $k$ such that every $k$-dichromatic digraph contains a subdivision of $F$.
	As our first main result, we prove that if $F$ is an orientation of a cycle then $\text{mader}_{\vec{\chi}}(F)=v(F)$. This settles a conjecture of Aboulker, Cohen, Havet, Lochet, Moura and Thomass\'{e}. We also extend this result to the more general class of orientations of cactus graphs, and to bioriented forests. 
	
	Our second main result is that $\text{mader}_{\vec{\chi}}(F)=4$ for every tournament $F$ of order $4$. This is an extension of the classical result by Dirac that $4$-chromatic graphs contain a $K_4$-subdivision to directed graphs.
\end{abstract}

\section{Introduction}
The chromatic number is one of the fundamental graph parameters, and is well-known to be intractable. Meaningful sufficient and necessary conditions for it to be large are of high interest. 
In fact, some of the most important results and open problems of Graph Theory are concerned with the relation between the chromatic number of an undirected graph and its containment of substructures such as subgraphs, minors or topological minors.
A prime example is the famous Hadwiger conjecture from 1943, which states the following:
\begin{conjecture}[\cite{hadwiger}]\label{hadwiger}
Every graph $G$ with $\chi(G) \ge k$ contains $K_k$ as a minor.
\end{conjecture}
An even stronger conclusion was suggested by Haj\'{o}s, who conjectured that every $k$-chromatic graph contains a {\em subdivision} of $K_k$, that is, a graph which can be obtained from $K_k$ by replacing its edges with pairwise internally vertex-disjoint paths connecting their original endpoints. Haj\'{o}s' conjecture is easily verified for $k \le 3$, and Dirac~\cite{dirac4} proved the case $k=4$.
\begin{theorem}[\cite{dirac4}]\label{thm:dirac4}
Every graph $G$ with $\chi(G)\ge 4$ contains a $K_4$-subdivision. 
\end{theorem}
While the cases $k=5,6$ of Haj\'{o}s' conjecture remain open, it was disproved for all values $k \ge 7$ by Catlin~\cite{catlin}, who constructed explicit counterexamples, i.e. graphs with chromatic number $k$ which contain no $K_k$-subdivision (see also~\cite{thom05}). An even more devastating blow to the conjecture was delivered by Erd\H{o}s and Fajtlowicz~\cite{erdosfajtlowicz}, who showed that {\em almost all} graphs on $\Theta(k^2)$ vertices do {\em not} contain a $K_k$-subdivision, even though their chromatic number is $\Omega\left(k^2/\log k\right)$. 

On the positive side, it turned out that large enough chromatic number does in fact necessitate the existence of a $K_k$-subdivision. As a matter of fact, the following classical result established that even large density is sufficient.
\begin{theorem}[Bollob\'{a}s and Thomason~\cite{bollobas}, Koml\'{o}s and Szemer\'{e}di~\cite{Komlos}]\label{thm:density}
	There exists an absolute constant $C>0$ such that for every $k \in \mathbb{N}$, every graph $G$ with minimum degree at least $Ck^2$ contains a subdivision of $K_k$.
\end{theorem}

Since every graph $G$ contains a subgraph of minimum degree at least $\chi(G)-1$,
one can deduce from Theorem~\ref{thm:density} that having chromatic number larger than $Ck^2$ (for some absolute constant $C$) is sufficient to guarantee a $K_k$-subdivision.
For $k \in \mathbb{N}$, let $f(k)$ be the smallest integer such that every graph with chromatic number at least $f(k)$ contains a $K_k$-subdivision.
Theorem~\ref{thm:density} then implies a quadratic upper bound $f(k)=O(k^2)$,
while the result of Erd\H{o}s and Fajtlowicz~\cite{erdosfajtlowicz} establishes a lower bound of $f(k)=\Omega\left(k^2/\log k\right)$.
These remain the best known bounds on $f(k)$;
Fox et al.~\cite{foxetal} conjectured that the truth lies with the lower
bound.

The upshot of the above discussion is that a subdivision of any given graph is contained in any graph of sufficiently large chromatic number. In this paper we investigate this phenomenon in the realm of directed graphs; we ask in what form, and to what extent, it holds.
The notion of subdivision extends naturally to directed graphs: given a digraph $F$, a \emph{subdivision} of $F$ is any digraph obtained by replacing every arc $(x,y)$ in $F$ by a directed path from $x$ to $y$, such that subdivision-paths of different arcs are internally vertex-disjoint.
It is less clear, however, how to choose a suitable chromatic number concept, which would provide a rich family of forcible digraph subdivisions.
The \emph{chromatic number} $\chi(D)$ of a digraph $D$ is defined as the chromatic number of the underlying graph of $D$. The fact that any graph, however high its chromatic number is, can be oriented acyclically and hence avoid containing any directed cycle, already hints that $\chi(D)$ being large might only have limited impact on digraph subdivision containment.
In fact, as was noted by Aboulker, Cohen, Havet, Lochet, Moura, and Thomass\'e~\cite{aboulker}, the family of digraphs $F$ which can be forced as a subdivision by high chromatic number is very limited: it consists of the orientations of forests. See~\cite{burr} and~\cite{chromorientedcycles}, respectively, for the positive and negative directions of this result. As a consequence, we see that high chromatic number of the underlying graph is not even strong enough to force the subdivision of any particular orientation of a cycle. 

Another widely-studied digraph coloring parameter --- which, in contrast to the chromatic number, takes into account the direction of edges --- is the \emph{dichromatic number}. Given a digraph $D$, an \emph{acyclic $k$-coloring} of $D$ is a mapping $c:V(D) \rightarrow [k]$ such that for every color $i \in [k]$, the \emph{color class} $c^{-1}(i) \subseteq V(D)$ induces an acyclic subdigraph of $D$. The \emph{dichromatic number} $\vec{\chi}(D)$ is defined as the smallest $k \in \mathbb{N}$ for which an acyclic $k$-coloring of $D$ exists. Introduced in 1982 by Neumann-Lara \cite{neulara}, this parameter was rediscovered and popularized by Mohar \cite{MoharEdgeWeight}, and since then has received further attention, see \cite{aboulker, perfect, lists, HARUTYUNYAN2019, dig5, dig4, fractionalNL} for some selected recent results.   

Aboulker~et~al.~\cite{aboulker} initiated the study of the existence of various subdivisions in digraphs of large dichromatic number.
In one of their main results, they show that a subdivision of any given digraph is contained in digraphs of sufficiently large dichromatic number.
\begin{theorem}[\cite{aboulker}, Theorem 32]\label{thm:dichromatic_clique}
	Let $F$ be a digraph with $n$ vertices and $m$ arcs. Then every digraph $D$ with $\vec{\chi}(D)\ge 4^m(n-1)+1$ contains a subdivision of $F$.
\end{theorem} 
Following the terminology in~\cite{aboulker}, for a digraph $F$ we denote by $\text{mader}_{\vec{\chi}}(F)$ the smallest integer $k \ge 1$ such that every digraph $D$ with $\vec{\chi}(D) \ge k$ contains a subdivision of $F$. We call $\text{mader}_{\vec{\chi}}(F)$ the {\em (dichromatic) Mader number of $F$}.

The problem of 
obtaining a polynomial bound (in terms of the number of vertices and arcs of $F$) on the Mader number remains open, and seems quite challenging. One reason for the increased difficulty 
compared to the undirected case is that there is no analogue of Theorem~\ref{thm:density} for directed graphs. In fact, it follows from a result of Thomassen~\cite{thom85} that there exist digraphs of arbitrarily high minimum out- and in-degree, which do not even contain a subdivision of $\bivec{K}_3$, the bioriented triangle. Consequently, entirely new methods have to be developed to force clique-subdivisions in digraphs of large dichromatic number, since any methods for addressing this problem must differ substantially from the established density-based ideas used in the undirected theory. 

%

In light of the difficulty of improving the bound in Theorem \ref{thm:dichromatic_clique} in general, it is natural to vie for obtaining better upper bounds 
for special classes of digraphs. One appealing conjecture in this vein was raised by 
Aboulker et al.~\cite{aboulker}. Let $C_{\ell}$ denote the (undirected) cycle of length $\ell$. 
\begin{conjecture}[\cite{aboulker}, Conjecture 39]\label{conj:cycles}
	If $C$ is an orientation of $C_\ell$, then $\text{mader}_{\vec{\chi}}(C)=\ell$.
\end{conjecture}
\noindent
In \cite{aboulker}, Conjecture \ref{conj:cycles} is proved for directed cycles and an upper bound of $2\ell-1$ is established for arbitrary orientations.

\subsection{Our results.}
Note that the Mader number of every digraph $F$ is at least the number of its vertices.
Indeed, the complete digraph of order $v(F)-1$ has dichromatic number $v(F)-1$, but does not have enough vertices to host a subdivision of $F$. Hence Conjecture~\ref{conj:cycles} states that, in a sense, the Mader number
of orientations of cycles is as small as it could be.

In this paper we resolve Conjecture \ref{conj:cycles} and go on to study the more general question: for which digraphs $F$ does it hold that $\text{mader}_{\vec{\chi}}(F) = v(F)$? In the first main result of our paper we prove that this equality holds for a large class of digraphs, which includes orientations of cactus graphs\footnote{Cactus graphs are usually defined as the graphs which do not contain a pair of cycles sharing at least two vertices (or, equivalently, as the graphs which do not contain $K_4 - e$ as a minor).
} (and hence all orientations of cycles), as well as all bioriented forests. This class of digraphs, a member of which we refer to as {\em octus}\footnote{The name alludes to the fact that every orientation of a cactus graph is an octus.
  We should warn, however, that the class of octi is strictly larger than the class of orientations of cacti, as is explained following Theorem \ref{thm:main1}.
}, is defined inductively as follows.

\begin{definition}\label{def:octi} The class of {\em octi digraphs} is defined as follows. 
%
	\begin{itemize}
		\item $K_1$ is an octus. 
		\item Let $F$ be an octus, and let $v_0 \in V(F)$. 
		Let $P = v_1,\ldots,v_k$, $k \geq 1$, be an orientation of a path which is disjoint from $V(F)$.
		Let $F^\ast$ be obtained from $F$ by adding the path $P$, both arcs $(v_0,v_1),(v_1,v_0)$, and exactly one of the arcs $(v_0,v_k),(v_k,v_0)$. Then $F^*$ is also an \nolinebreak octus.  
		\item If $F$ is an octus then every subdigraph of $F$ is also an octus. 
	\end{itemize} 
\end{definition}
We note that the path $P$ in the second item of Definition \ref{def:octi} is allowed to consist of a single vertex, which corresponds to attaching a digon to $F$ at $v_0$. The operation described in Item 2 of Definition \ref{def:octi} will be called {\em ear addition}. 
Our first main result is as follows: 
\begin{theorem}\label{thm:main1}
	For every octus $F$, we have $\text{mader}_{\vec{\chi}}(F)=v(F)$. 
\end{theorem}

This theorem has a couple of immediate consequences, each of which extends results of \cite{aboulker}. 
It is not difficult to see that orientations of cacti are precisely the octi which have no digons. 
Therefore, we have the following:
\begin{corollary}\label{cor:cactus}
	For every orientation $F$ of a cactus, we have $\text{mader}_{\vec{\chi}}(F)=v(F)$. 
\end{corollary}
Since every cycle is a cactus, Corollary~\ref{cor:cactus} immediately implies Conjecture \ref{conj:cycles}.

Another immediate corollary of Theorem~\ref{thm:main1} is concerned with biorientations of forests. 
Here, a {\em biorientation} of an undirected graph $G$ is the (symmetric) digraph $\bivec{G}$ obtained by replacing each edge $\{x,y\}$ of $G$ with the arcs $(x,y),(y,x)$. Every bioriented tree can be obtained from $K_1$ by a sequence of ear additions where, at each step, we add a new vertex and connect it by a digon to one of the vertices of the existing digraph. Hence, every biorientation of a forest is an octus, and we have the following:
\begin{corollary}\label{cor:biforest}
	If $T$ is an (undirected) forest, then $\text{mader}_{\vec{\chi}}(\bivec{T})=v(\bivec{T})$.
\end{corollary}
Corollary~\ref{cor:biforest} strengthens another result from~\cite{aboulker}, where the conclusion was shown to hold for every orientation of a forest.

Next we discuss digraphs on a small number of vertices. The smallest digraph not covered by Theorem~\ref{thm:main1} is the bioriented triangle minus an edge. It turns out that this digraph, too, has the property that its Mader number equals its number of vertices. 
\begin{proposition}\label{K_3-e}
	$\text{mader}_{\vec{\chi}}(\bivec{K}_3 - e) = 3$.
\end{proposition}

In the second main result of our paper, we show that the Mader number of every $4$-vertex tournament is $4$. 
\begin{theorem}\label{thm:dirdirac}
	For every orientation $K$ of $K_4$, we have that $\text{mader}_{\vec{\chi}}(K) = 4$. 
\end{theorem}
Theorem \ref{thm:dirdirac} is a strict extension to the directed setting of Dirac's theorem on $K_4$-subdivisions (namely, Theorem~\ref{thm:dirac4}). In fact, Theorem~\ref{thm:dirac4} can be easily derived from Theorem \ref{thm:dirdirac} as follows. 
First, observe that $\vec{\chi}(\bivec{G})=\chi(G)$ for every graph $G$. 
Now, if $G$ is an undirected graph with $\chi(G) \ge 4$, then by Theorem~\ref{thm:dirdirac}, $\bivec{G}$ contains a subdivision of any orientation of $K_4$, which translates to a $K_4$-subdivision in $G$.

The rest of this paper is organized as follows. After establishing some preliminary results in Section~\ref{sec:prelim}, we prove Theorem \ref{thm:main1} in Section~\ref{sec:octi}. Section~\ref{sec:tourns} is devoted to proving Theorem \ref{thm:dirdirac}. Finally,
in Section~\ref{sec:complete} we conclude with a discussion of Mader numbers of biorientations of complete digraphs and cycles, give the proof of Proposition \ref{K_3-e}, and pose some open problems. A main focus of Section \ref{sec:complete} is on digraphs which we call {\em Mader-perfect}; these are digraphs $F$ with the property that every subdigraph $F'$ of $F$ satisfies $\text{mader}_{\vec{\chi}}(F')=v(F')$. We propose the further study of these digraphs and establish some preliminary results. 
\paragraph*{Notation.}
All digraphs considered in this paper are loopless, have no parallel edges, but are allowed to have anti-parallel pairs of edges (\emph{digons}). A directed edge (also called an arc) with tail $u$ and head $v$ is denoted by $(u,v)$. For a graph $G$, we denote by $V(G)$, $E(G)$ the vertex- and edge-set of $G$, respectively. For a digraph $D$, $V(D)$ denotes the vertex-set and $A(D)$ denotes the set of arcs; we will use the notation $v(D) = |V(D)|$ and $a(D) = |A(D)|$. For $X \subseteq V(D)$ we denote by $D[X]$ the induced subdigraph of $D$ with vertex-set $X$. For a set $X$ of vertices or arcs in $D$, we denote by $D-X$ the subdigraph obtained by deleting the objects in $X$ from $D$.
Given an undirected simple graph $G$, an \emph{orientation} of $G$ is any digraph obtained by replacing each edge $\{u,v\}$ of $G$ with (exactly) one of the arcs $(u,v)$ or $(v,u)$. Evidently, any orientation is digon-free. 
For a digraph $D$ and a vertex $v \in V(D)$, we let $N^+(v),$ $N^-(v)$ denote the out- and in-neighborhood of $v$ in $D$, and $d^+(v)$, $d^-(v)$ their respective sizes. We denote by $\delta^+(D)$, $\delta^-(D)$, $\Delta^+(D)$, $\Delta^-(D)$ the minimum or maximum  out- or in-degree of $D$, respectively. 
We use the words ``path" and ``cycle" to mean an orientation of a path or a cycle (respectively). For example,
a path $P$ in a digraph $D$ is an alternating sequence $v_1,e_1,v_2,\ldots,v_{k-1},e_{k-1},v_k$ of pairwise distinct vertices $v_1,\ldots,v_k \in V(D)$ and arcs $e_1,\ldots,e_{k-1} \in A(D)$ such that $e_i$ connects $v_{i}$ and $v_{i+1}$ (i.e., either $e_i = (v_i,v_{i+1})$ or $e_i = (v_{i+1},v_i)$). If in addition $e_i=(v_{i},v_{i+1})$ for every $i=1,\ldots,k-1$, then we say that $P$ is a \emph{directed path} or \emph{dipath} from $v_1$ to $v_k$ (a {\em $v_1$,$v_k$-dipath} for short). Given two distinct vertices $x \neq y$ on a path $P$, we denote by $P[x,y]=P[y,x]$ the subpath of $P$ with endpoints $x$ and $y$. 
A \emph{directed cycle} ({\em dicycle} for short) is a cycle with all arcs oriented consistently in one direction. 
For a {\em directed} cycle $C$ and two distinct vertices $x,y \in V(C)$, we denote by $C[x,y]$ the segment of $C$ which forms a dipath from $x$ to $y$ (note that $C[x,y] \neq C[y,x]$). A \emph{closed directed walk} is an alternating sequence $v_0,e_0,v_1,\ldots,v_{k-1},e_{k-1},v_k=v_0$ of vertices and arcs such that $e_i=(v_i,v_{i+1})$ for all $i$. 
A digraph $D$ is called \emph{weakly connected} (or just {\em connected}) if every two vertices can be connected by a path (i.e., if the underlying undirected graph is connected); and it is called \emph{strongly connected} if for every ordered pair $(x,y) \in V(D) \times V(D)$, there exists an $x$,$y$-dipath in $D$. The maximal strongly connected subgraphs of a digraph $D$ induce a partition of $V(D)$ and are called the \emph{strong components} of $D$. For a natural number $k \in \mathbb{N}$, a digraph is called \emph{strongly $k$-vertex-connected} (resp. \emph{strongly $k$-arc-connected}) if for every subset $K$ of at most $k-1$ vertices (resp. arcs), the digraph $D-K$ is strongly connected. The notions of {\em weak $k$-vertex-connectivity} and {\em weak $k$-arc-connectivity} are defined analogously. An {\em in-arborescence} is a directed rooted tree in which all arcs are directed towards the root. 

\section{Preliminaries}\label{sec:prelim}
In this section we gather a number of definitions, observations and auxiliary results about the dichromatic number and about subdivisions in digraphs which will be used in the course of the paper.
We start by observing that $\text{mader}_{\vec{\chi}}$ is subadditive with respect to taking disjoint unions.  
\begin{observation}\label{subadditivity}
	Let $F$ be the disjoint union of two digraphs $F_1,F_2$. Then $$\text{mader}_{\vec{\chi}}(F) \leq \text{mader}_{\vec{\chi}}(F_1) + \text{mader}_{\vec{\chi}}(F_2).$$
\end{observation}
\begin{proof}
	For convenience, put $k_i := \text{mader}_{\vec{\chi}}(F_i)$, $i = 1,2$. Let $D$ be a digraph with dichromatic number at least $k_1 + k_2$. Let $A_1 \subseteq V(D)$ be such that $\text{mader}_{\vec{\chi}}(D[A_1]) = k_1$ (such a set $A_1$ can be obtained by repeatedly deleting vertices as long as the dichromatic number of the current digraph is strictly larger than $k_1$). Put $A_2 := V(D) \setminus A_1$. Then $\text{mader}_{\vec{\chi}}(D[A_2]) \geq k_2$, for otherwise one could color $D$ with less than $k_1 + k_2$ colors. By our choice of $k_i$, we get that $D[A_i]$ contains a subdivision of $F_i$ for each $i = 1,2$. It follows that $D$ contains a subdivision of $F$, as required.  
\end{proof}

Let $k \in \mathbb{N}$. A digraph $D$ is called \emph{$k$-dicritical}, if $\vec{\chi}(D)=k$, but $\vec{\chi}(D')<k$ for all proper subdigraphs $D' \subsetneq D$.
\begin{lemma}\label{mindegree}
Let $D$ be $k$-dicritical. Then $\delta^+(D), \delta^-(D) \ge k-1$.
\end{lemma}
\begin{proof}
Since the reversal of all arcs preserves the $k$-dicriticality of $D$, it suffices to show that $\delta^+(D) \ge k-1$. Suppose towards a contradiction that there exists some $v \in V(D)$ such that $d^+(v)<k-1$. By assumption, $D-v$ admits an acyclic coloring with color-set $\{1,\ldots,k-1\}$. We can extend this to a $(k-1)$-coloring of $D$ by assigning to $v$ a color in $\{1,\ldots,k-1\}$ that does not appear on $N^+(v)$. Then the resulting coloring is an acylic $(k-1)$-coloring of $D$ (since no monochromatic directed cycle can pass through $v$), in contradiction to our assumption \nolinebreak that \nolinebreak $\vec{\chi}(D)=k$.
\end{proof}

\begin{lemma}\label{strong connectivity}
	Let $D$ be $k$-dicritical. Then $D$ is strongly connected.
\end{lemma}
\begin{proof}
	Assume, for the sake of contradiction, that $D$ is not strongly connected. Then there is a partition $V(D) = A \cup B$ such that $A$ and $B$ are non-empty and there are no arcs going from $B$ to $A$. Since $D$ is $k$-dicritical, both $D[A]$ and $D[B]$ have an acyclic $(k-1)$-coloring. But putting these colorings together is an acyclic $(k-1)$-coloring of $D$, since $D$ contains no directed cycles which intersect both $A$ and $B$. Thus, we have arrived at a contradiction to $\vec{\chi}(D)=k$.
      \end{proof}

Given a digraph $D$ and two (not necessarily disjoint) subsets $A, B \subseteq V(D)$, an \emph{$A$-$B$-dipath} is a directed path in $D$ which starts in a vertex of $A$, ends in a vertex of $B$, and is internally vertex-disjoint from $A \cup B$ (here we allow paths consisting of a single vertex belonging to $A \cap B$). Similarly, for a vertex $u \in V(D)$, by a {\em $u$-$A$-dipath} or an {\em $A$-$u$-dipath}, respectively, we mean a $\{u\}$-$A$ or an $A$-$\{u\}$-dipath according to the above definition.
We will frequently use the following well-known variants of Menger's Theorem for directed graphs.
\begin{theorem}\label{setmenger}
Let $D$ be a digraph and $k \in \mathbb{N}$.
\begin{itemize}
\item[(i)] If $D$ is strongly $k$-vertex-connected, then for any two subsets $A, B \subseteq V(D)$ such that $|A|,|B| \ge k$, there are $k$ pairwise vertex-disjoint $A$-$B$-paths. 
\item[(ii)] If $v \in V(D)$ and $A \subseteq V(D) \setminus \{v\}$, then either there are $k$ different $v$-$A$-dipaths which pairwise only intersect at $v$, or there is a subset $K \subseteq V(D) \setminus \{v\}$ such that $|K|<k$ and such that in $D-K$ there is no dipath starting at $v$ and ending in $A$. 
\end{itemize}
\end{theorem}
\begin{proof}\noindent
For (i), note that in a strongly $k$-vertex connected digraph the smallest set that intersects every $A$-$B$-path is of size at least $k$. Hence the existence of a family of $k$ pairwise vertex-disjoint $A$-$B$-paths is implied by Menger's Theorem~\cite{mengeroriginal} (cf.~\cite{Goring}).

Let us now derive (ii). If there exists no family of $k$ different $v$-$A$-dipaths that only intersect in $v$, then there is no family of $k$ disjoint $N^+(v)$-$A$-dipaths in $D-v$. Menger's Theorem then ensures the existence of a set $K\subseteq V(D)\setminus\{ v \}$ of less than $k$ vertices such that there is no $N^+(v)$-$A$-dipath in $D-(\{v\} \cup K)$. Then in $D-K$ there is no $v$-$A$-dipath, since all such dipaths go through $N^+(v)$.
\end{proof}
We will further need the following two deep results by Mader on so-called non-critical vertices and on subdivisions in digraphs of sufficiently large out-degree.
\begin{theorem}[\cite{Mader_critical_connectivity}, see also Section 7.11 in~\cite{BJ-G}]\label{critical strongly connected}
Let $k \in \mathbb{N}$, and let $D$ be a strongly $k$-vertex-connected digraph with $\delta^+(D),\delta^-(D) \geq 2k$. Then there is $v \in V(D)$ such that $D - v$ is (also) strongly $k$-vertex-connected. 
\end{theorem}
\begin{theorem}[\cite{Mader_trans_4}]\label{subdivtrans4}
Let $D$ be a digraph such that $\delta^+(D) \ge 3$. Then $D$ contains a subdivision of $\vec{K}_4$, the transitive tournament of order $4$.
\end{theorem}

\section{Oriented cacti and bioriented forests}\label{sec:octi}
In this section we prove Theorem~\ref{thm:main1}. 
The main step in the proof consists of showing that if $F^\ast$ is a digraph obtained from a digraph $F$ via ear addition (i.e., the operation described in the second item of Definition \ref{def:octi}),
then $\text{mader}_{\vec{\chi}}(F^\ast) \le \text{mader}_{\vec{\chi}}(F)+k$ where $k$ is the number of newly added vertices. This is done in the \nolinebreak following \nolinebreak theorem. 

\begin{theorem}\label{kempe}
	Let $F$ be a digraph and let $v_0 \in V(F)$. 
	Let $P = v_1,\ldots,v_k$ be an orientation of a path disjoint from $V(F)$.
	Let $F^\ast$ be a digraph obtained from $F$ by adding the path $P$, both arcs $(v_0,v_1),(v_1,v_0)$, and exactly one of the arcs $(v_0,v_k),(v_k,v_0)$. 
	Then $\text{mader}_{\vec{\chi}}(F^\ast) \le \text{mader}_{\vec{\chi}}(F)+k$.
\end{theorem}
 

To prove Theorem \ref{kempe} we will need the following useful lemma, which describes a generalization of the idea of Kempe-switches to directed graphs.
\begin{lemma}\label{componentswitch}
Let $D$ be a digraph, $k \in \mathbb{N}$, and let $c:V(D) \rightarrow \{1,\ldots,k\}$ be an acyclic coloring of $D$. Let $i \neq j \in \{1,\ldots,k\}$, $D_{i,j}:=D[c^{-1}(\{i,j\})]$, and let $X \subseteq c^{-1}(\{i,j\})$ be the vertex set of a strong component of $D_{i,j}$. Then the coloring $c':V(D) \rightarrow \{1,\ldots,k\}$, defined by
$$c'(x):=\begin{cases} c(x) & \text{ if }x \in V(D) \setminus X, \cr j & \text{ if }x \in X \cap c^{-1}(i), \cr i & \text{ if }x \in X \cap c^{-1}(j) \end{cases}$$ is an acyclic coloring of $D$ as well.
\end{lemma}
\begin{proof}
  Suppose towards a contradiction that there is a directed cycle $C$ in $D$ which is monochromatic under $c'$. If $V(C) \cap X=\emptyset$, then $c$ and $c'$ agree on $V(C)$, contradicting our assumption that $c$ is an acyclic coloring of $D$. Therefore $V(C) \cap X \neq \emptyset$. Since $c'$ has only colors $i$ or $j$ on $X$, we find that $C$ is monochromatic under $c'$ either in color $i$ or $j$. This means that $V(C) \subseteq (c')^{-1}(\{i,j\})=c^{-1}(\{i,j\})$ according to the definition of $c'$. Hence, $C$ is a directed cycle in $D_{i,j}$, and since $X$ is a strong component of $D_{i,j}$, we conclude $V(C)\subseteq X$. By the definition of $c'$ the colors $i$ and $j$ are switched in $X$, so $C$ must have been monochromatic under $c$ in color $j$ or $i$. This contradicts to the fact that the coloring $c$ of $D$ is acyclic and concludes the proof. 
\end{proof}
\begin{proof}[Proof of Theorem \ref{kempe}]

First, we argue that by symmetry, it is enough to handle the case that $(v_0,v_k) \in A(F^\ast)$. For a digraph $D$, denote by $\revvec{D}$ the digraph obtained from it by reversing the orientations of all arcs, that is, $V(\revvec{D}):=V(D)$, $A(\revvec{D}):=\{(x,y) \; | \; (y,x) \in A(D)\}$. For all digraphs $D$ and $F$, we have $\vec{\chi}(D)=\vec{\chi}(\revvec{D})$ and $D$ contains a subdivision of $F$ if and only if $\revvec{D}$ contains a subdivision of $\revvec{F}$. As a consequence, we have $\text{mader}_{\vec{\chi}}(F)=\text{mader}_{\vec{\chi}}(\revvec{F})$ for every digraph $F$. Therefore, the case $(v_k,v_0) \in A(F^\ast)$ follows from the case $(v_0,v_k) \in A(F^\ast)$ via this symmetry. So for the rest of the proof we assume that $(v_0,v_k) \in A(F^\ast)$.

For brevity, in the following we put $M:=\text{mader}_{\vec{\chi}}(F)$.
Consider any given digraph $D$ such that $\vec{\chi}(D)=M+k$. We have to show that $D$ contains a subdivision of \nolinebreak $F^\ast$. 

Let us start by fixing an acyclic coloring $c_0:V(D) \rightarrow \{1,2,\ldots,M+k\}$ of $D$ that maximizes $|c_0^{-1}(\{1, \ldots, k\})|$. 
In the following, we set $Y_1:=c_0^{-1}(\{1,\ldots,k\})$ and  $Y_2:=c_0^{-1}(\{k+1,\ldots,M+k\})$. Note that $V(D)=Y_1 \cup Y_2$ is a partition of $V(D)$.
Since $c_0$ is an acyclic coloring of $D$ with $\vec{\chi}(D)$ colors, we have $\vec{\chi}(D[Y_1])=|\{1,\ldots,k\}| = k$ and $\vec{\chi}(D[Y_2])=|\{k+1,\ldots,M+k\}|= M$. 

From the definition of $M$ we conclude that there exists a subgraph $S \subseteq D[Y_2]$ which is a subdivision of $F$. In the following, let us denote by $x_0 \in V(S) \subseteq Y_2$ the vertex in this subdivision corresponding to $v_0 \in V(F)$. 

For each acyclic $k$-coloring $c:Y_1 \rightarrow \{1,\ldots,k\}$ of $D[Y_1]$, let $\textbf{v}(c) \in \mathbb{Z}^{k}$ denote the vector defined by $\textbf{v}(c)_{i}=|N^+(x_0) \cap c^{-1}(i)|$, for $i=1,\ldots,k$. Let us consider the pre-order $\prec$ on the set of acyclic $\{1,\ldots,k\}$-colorings of $D[Y_1]$, where $c_1 \prec c_2$ iff $\textbf{v}(c_1) <_{\text{lex}} \textbf{v}(c_2)$. Here $<_{\text{lex}}$ denotes the lexicographical order on $\mathbb{Z}^{k}$. 
In the following, let $c:Y_1 \rightarrow \{1,\ldots,k\}$ denote an acyclic coloring of $D[Y_1]$ that is minimal with respect to $\prec$. For $i<j \in \{1,\ldots,k\}$, let $D_{i,j}:=D[c^{-1}(\{i,j\})]$.

\paragraph{Claim 1.} For every $1 \le i <j \le k$ and every vertex $x \in N^+(x_0) \cap c^{-1}(i)$, there is a vertex $y \in N^+(x_0) \cap c^{-1}(j)$ such that $x$ and $y$ lie in the same strong component of $D_{i,j}$.
\begin{proof}
Denote by $X \subseteq c^{-1}(\{i,j\})$ the unique strong component of $D_{i,j}$ containing $x$. Suppose towards a contradiction that $X \cap (N^+(x_0) \cap c^{-1}(j))=\emptyset$. Let $c'$ be the coloring of $D[Y_1]$ obtained from $c$ by switching colors $i$ and $j$ within $X$. According to Lemma~\ref{componentswitch}, $c'$ is an acyclic coloring of $D[Y_1]$. By definition, we furthermore have $\textbf{v}(c')_\ell=\textbf{v}(c)_\ell$ for all $\ell \in \{1,\ldots,k\}\setminus\{i,j\}$, and since no vertex in $N^+(x_0)$ is switched from color $j$ to color $i$ while $x$ is switched from color $i$ to color $j$, we have $\textbf{v}(c')_i<\textbf{v}(c)_i$. However, since $i<j$, this means that $c'\prec c$, contradicting our minimality assumption on $c$. This shows that our assumption was wrong, namely there does exist a vertex $y \in X \cap (N^+(x_0) \cap c^{-1}(j))$. This yields the claim.
\end{proof}

\paragraph{Claim 2.} There are vertices $x_1,x_2,\ldots,x_k \in N^+(x_0) \cap Y_1$ such that
\begin{itemize}
\item $c(x_i)=i$, for $i=1,\ldots,k$.
\item There is a directed cycle $C$ in $D$ containing $x_0$ and $x_1$ such that $V(C) \setminus \{x_0\} \subseteq c^{-1}(1)$. 
\item For every $2 \le i \le k$, there exists a directed path $P_{i-1,i}$ in $D[Y_1]$ with endpoints $x_{i-1},x_{i}$ such that $V(P_{i-1,i}) \subseteq c^{-1}(\{i-1,i\})$. In addition, $P_{i-1,i}$ is directed from $x_{i-1}$ to $x_{i}$ if $(v_{i-1},v_{i}) \in A(P)$, and directed from $x_{i}$ to $x_{i-1}$ if $(v_{i},v_{i-1}) \in A(P)$.
\end{itemize}
\begin{proof}
We start by showing 
that there is a directed cycle $C$ in $D$ through $x_0$ such that \linebreak $V(C) \setminus \{x_0\} \subseteq c^{-1}(1)$. Assume, towards a contradiction, that no such cycle exists, and consider the coloring $c_0':V(D) \rightarrow \{1,\ldots,M+k\}$ defined by  
$$c_0'(x):=\begin{cases} c(x), & \text{ if }x \in Y_1 \cr 1, & \text{ if }x=x_0 \cr c_0(x), & \text{ if }x \in Y_2\setminus \{x_0\}. \end{cases}$$ 
Our assumption implies that $c_0'$ is an acyclic coloring of $D$, because there is no directed cycle containing $x_0$ which is monochromatic under $c_0'$. However, the coloring $c_0'$ has one more vertex in colors $\{1,...,k\}$ than $c_0$, contradicting our maximality assumption on $c_0$. Therefore, a cycle $C$ with the claimed properties exists.

Now define $x_1 \in N^+(x_0) \cap V(C)$ to be the unique out-neighbor of $x_0$ on $C$. We have $c(x_1)=1$ since $x_1\in V(C)\setminus \{ x_0\}$. We now successively define vertices $x_2,\dots,x_k$ as follows: for $i=2,3,\ldots,k$, define the vertex $x_i$ to be a vertex in $N^+(x_0) \cap c^{-1}(i)$ chosen such that $x_{i-1}$ and $x_i$ lie in the same strong component of $D_{i-1,i}$. Such a choice is possible by Claim 1. 

The first and second items of the claim follow directly from our choice of the vertices $x_1,\ldots,x_k$. For the last item, for each $2 \le i \le k$, we choose a directed path $P_{i-1,i}$ in $D_{i-1,i}$, such that $P_{i-1,i}$ is directed from $x_{i-1}$ to $x_i$ if $(v_{i-1},v_i) \in A(P)$ and from $x_i$ to $x_{i-1}$ if $(v_i,v_{i-1}) \in A(P)$. The existence of such a path follows in each case since $x_{i-1},x_i$ are in the same strong component of $D_{i-1,i}$. Clearly, $V(P_{i-1,i}) \subseteq V(D_{i-1,i})=c^{-1}(\{i-1,i\})$. This proves the last item.
\end{proof}

\paragraph{Claim 3.} There are vertices $z_1,z_2,\ldots,z_k \in Y_1$ such that
\begin{itemize}
\item $c(z_i)=i$, for $i=1,\ldots,k$.
\item $z_1 \in V(C)$ and $z_k \in N^+(x_0)$.
\item For every $2 \le i \le k$, there exists a directed path $Q_{i-1,i}$ in $D[Y_1]$ with endpoints $z_{i-1},z_{i}$ such that $Q_{i-1,i}$ is directed from $z_{i-1}$ to $z_{i}$ if $(v_{i-1},v_{i}) \in A(P)$, and directed from $z_{i}$ to $z_{i-1}$ if $(v_{i},v_{i-1}) \in A(P)$. 
\item The paths $Q_{i-1,i}, i=2,\ldots,k$ are pairwise internally vertex-disjoint. 
\item $V(C) \cap V(Q_{1,2})=\{z_1\}$ and $V(C) \cap V(Q_{i-1,i})=\emptyset$ for $i=3,\ldots,k$.
\end{itemize}
\begin{proof}
We define the vertices $z_i$ as follows: We define $z_1 \in V(C)$ to be the unique last vertex in $V(C)$ we meet when traversing the trace of the path $P_{1,2}$ starting from $x_1 \in V(C)$. Since $P_{1,2}$ uses only colors $1$ and $2$, we must have $z_1 \in V(C) \setminus \{x_0\}$ and thus $c(z_1)=1$. For $i=2,\ldots,k-1$, we successively define $z_i$ to be the first vertex of $P_{i,i+1}$ we meet when traversing the trace of the path $P_{i-1,i}[z_{i-1},x_i]$ starting from $z_{i-1}$ (such a vertex exists, since $x_i \in V(P_{i-1,i}) \cap V(P_{i,i+1})$ by Claim~2). Since $V(P_{i-1,i}) \subseteq c^{-1}(\{i-1,i\}), V(P_{i,i+1}) \subseteq c^{-1}(\{i,i+1\})$, it follows that $c(z_i)=i$. Finally, we put $z_k:=x_k \in N^+(x_0)$. For each $i \in \{2,3,\ldots,k\}$, we define $Q_{i-1,i}:=P_{i-1,i}[z_{i-1},z_i]$.

Let us now verify the correctness of the claim. The first three items follow directly from Claim~2 and the definition of the vertices $z_i$ and the paths $Q_{i-1,i}$.

For the fourth item, let $i<j \in \{2,\ldots,k\}$ be given. We need to show that $Q_{i-1,i}$ and $Q_{j-1,j}$ can only intersect in their endpoints.  If $j-i \ge 2$, then we directly conclude that $V(Q_{i-1,i}) \cap V(Q_{j-1,j}) \subseteq V(P_{i-1,i}) \cap V(P_{j-1,j}) \subseteq c^{-1}(\{i-1,i\}) \cap c^{-1}(\{j-1,j\})=\emptyset$. If on the other hand $j=i+1$, then by definition of the vertex $z_i$, no vertex on the path $Q_{i-1,i}=P_{i-1,i}[z_{i-1},z_i]$ except for $z_i$ lies on $P_{i,i+1}$, and therefore also not on $Q_{j-1,j}=P_{i,i+1}[z_i,z_{i+1}]$. Hence, $V(Q_{i-1,i}) \cap V(Q_{j-1,j}) =\{z_i\}=\{z_{j-1}\}$. This concludes the proof of the fourth item. 
The claim that $V(C) \cap V(Q_{1,2})=\{z_1\}$ in the fifth item directly follows from our choice of $Q_{1,2}=P_{1,2}[z_1,z_2]$ and the definition of $z_1$ as being the last vertex on $C$ we meet when traversing $P_{1,2}$ starting at $x_1$. For $i \in \{3,\ldots,k\}$, we can conclude the second part of the last item from the inclusion $V(C) \cap V(Q_{i-1,i}) \subseteq (c^{-1}(\{1\}) \cup \{x_0\}) \cap c^{-1}(\{i-1,i\})=\emptyset$. 
\end{proof}

Let $S^\ast$ be the subdigraph of $D$ formed by joining $S \subseteq D[Y_2]$, the pairwise distinct vertices $z_1,\ldots,z_k$ and the connecting dipaths $Q_{i-1,i}, i=2,\ldots,k$, the two anti-parallel directed paths $C[x_0,z_1]$, $C[z_1,x_0]$ between $x_0$ and $z_1$ as well as the arc $(x_0,z_k)$. From Claim 3 and since $\left(\bigcup_{i=2}^{k}{V(Q_{i-1,i})} \cup (V(C) \setminus \{x_0\})\right) \cap V(S)\subseteq Y_1 \cap Y_2=\emptyset$, it follows that $S^\ast$ is isomorphic to a subdivision of $F^\ast$, with $x_0,z_1,z_2,\ldots,z_k$ playing the roles of the vertices $v_0,v_1,v_2,\ldots,v_k$ of $F^\ast$. 

We have thus shown that every digraph $D$ with $\vec{\chi}(D)=\text{mader}_{\vec{\chi}}(F)+k$ contains a subdivision of $F^\ast$, and this concludes the proof of the theorem.
\end{proof}

By definition, every octus is obtained from $K_1$ via a sequence of operations of two types: ear addition and taking a subdigraph. For an octus $F$, let $s(F)$ be the (minimal) number of operations needed to obtain $F$. 
Let us say that $F$ is a {\em maximal octus} if it can be obtained from $K_1$ by a sequence of ear additions only. By repeatedly applying Theorem \ref{kempe}, we see that $\text{mader}_{\vec{\chi}}(F) = v(F)$ for every maximal octus $F$. To complete the proof of Theorem \ref{thm:main1}, we also need to address non-maximal octi. This will be done using the following two lemmas. 


\begin{lemma}\label{lem:octus_subdigraph_1}
	Every octus is a subdigraph of a maximal octus. 
\end{lemma}
\begin{proof}
	The proof is by induction on $s(F)$. If $s(F) = 0$ then $F = K_1$ and the assertion is trivial. Suppose then that $s(F) \geq 1$. By the definition of an octus (see Definition \ref{def:octi}), either $F$ is a subdigraph of some octus $F'$ with $s(F') < s(F)$, or $F$ is obtained by ear addition from some octus $F'$ with $s(F') < s(F)$. In the former case, the induction hypothesis implies that $F'$ --- and hence also $F$ --- is a subdigraph of a maximal octus, as required. Suppose then that $F$ is obtained by ear addition from some octus $F'$ with $s(F') < s(F)$. By the induction hypothesis, $F'$ is a subdigraph of some maximal octus $F''$. By performing on $F''$ the same ear addition which turns $F'$ into $F$, we obtain a maximal octus which contains $F$. This completes the proof.  
\end{proof}

\begin{lemma}\label{lem:octus_subdigraph_2}
	For every connected subdigraph $F'$ of a maximal octus $F$, there is a maximal octus $\bar{F}$ such that $F'$ is a {\em spanning} subdigraph of $\bar{F}$.
\end{lemma}
\begin{proof}
	The proof is by induction on $s(F)$. If $s(F) = 0$ then $F = K_1$ and the assertion is trivial. Let then $F$ be a maximal octus with $s(F) \geq 1$, and let $F'$ be a connected subdigraph of $F$. 
	By the definition of maximal octi, there is some maximal octus $F_0$ with $s(F_0) < s(F)$ and $v_0 \in V(F_0)$ such that $F$ is obtained from $F_0$ by ear addition, namely, by adding an oriented path $P=v_1,\dots,v_k$ with $\{v_1,\dots,v_k\} \cap V(F_0) = \emptyset$, as well as the arcs $(v_0,v_1),(v_1,v_0)$ and (w.l.o.g.) $(v_0,v_k)$. Consider the subdigraph $F'_0 := F'[V(F') \cap V(F_0)]$ of $F_0$. 
	If $V(F') \cap V(F_0) = \emptyset$, namely if $V(F') \subseteq \{v_1,\dots,v_k\}$, then $F'$ is an oriented path, and is hence a spanning subgraph of a bioriented path, which is a maximal octus. Suppose then that $V(F') \cap V(F_0) \neq \emptyset$. 
	The way $F$ is constructed from $F_0$ and the assumption that $F'$ is connected imply that $F'_0$ is connected as well. 
	By the induction hypothesis (applied to $F_0$), there is a maximal octus $\bar{F_0}$ such that $F'_0$ is a spanning subdigraph of $\bar{F_0}$. If $F' = F'_0$ then we are done, and otherwise we must have $V(F') \cap \{v_1,\dots,v_k\} \neq \emptyset$, which in turn implies that $v_0 \in V(F'_0) = V(\bar{F_0})$ because $F'$ is connected. Now, if $\{v_1,\dots,v_k\} \subseteq V(F')$ then $F'$ is a spanning subdigraph of the maximal octus obtained from $V(\bar{F_0})$ by adding the path $P$ and connecting its endpoints to $v_0 \in V(\bar{F_0})$ using the arcs $(v_0,v_1),(v_1,v_0),(v_0,v_k)$. Otherwise, i.e. if $\{v_1,\dots,v_k\} \not\subseteq V(F')$, then there must be some $1 \leq i < j \leq k$ such that $V(F') = V(\bar{F_0}) \cup \{v_1,\dots,v_i\} \cup \{v_j,\dots,v_k\}$ (as $F'$ is connected). Now, let $\bar{F}$ be the maximal octus obtained from $\bar{F_0}$ by a sequence of two ear additions: we first add the path $v_1,\dots,v_i$ and the arcs $(v_0,v_1),(v_1,v_0),(v_0,v_i)$ and then the path $v_j,\dots,v_k$ and the arcs $(v_0,v_j),(v_j,v_0),(v_0,v_k)$. Then $F'$ is a spanning subdigraph of $\bar{F}$, as required.  	
\end{proof}

\begin{proof}[Proof of Theorem \ref{thm:main1}]
	Our goal is to show that $\text{mader}_{\vec{\chi}}(F) = v(F)$ for every octus $F$. First, observe that it suffices to prove this statement for connected $F$, since the general statement would then follow by invoking Observation \ref{subadditivity}.
	So let $F$ be a connected octus. By combining Lemmas \ref{lem:octus_subdigraph_1} and \ref{lem:octus_subdigraph_2}, we see that $F$ is a spanning subdigraph of some maximal octus $\bar{F}$. As mentioned before, Theorem \ref{kempe} implies that $\text{mader}_{\vec{\chi}}(\bar{F}) = v(\bar{F}) = v(F)$. As $F$ is a subdigraph of $\bar{F}$, we have $\text{mader}_{\vec{\chi}}(F) \leq \text{mader}_{\vec{\chi}}(\bar{F})$ and hence 
	$\text{mader}_{\vec{\chi}}(F) = v(F)$, as required.  
\end{proof}

\section{Tournaments of order $4$}\label{sec:tourns}
In this section we prove Theorem~\ref{thm:dirdirac}. 
We give a separate proof for each of the $4$-vertex tournaments. There are exactly four non-isomorphic tournaments on $4$ vertices: $\vec{K}_4$, the transitive tournament of order $4$; $\vec{K}_4^s$, the unique strongly connected tournament of order $4$; and the tournaments $W_4^+, W_4^-$ obtained from the directed triangle $\vec{C}_3$ by adding a dominating source or sink, respectively. See Figure~\ref{fig:4tourns} for an illustration. 

\begin{figure}[h]
\centering
\includegraphics[scale=0.6]{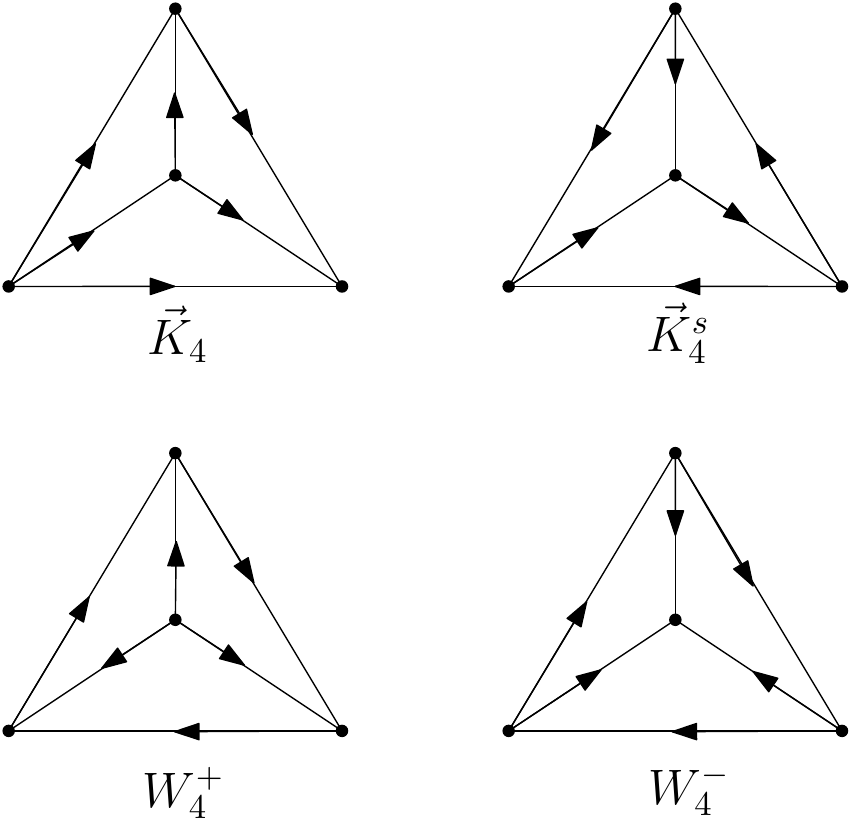}
\caption{The four non-isomorphic tournaments of order $4$.}\label{fig:4tourns}
\end{figure}

Since $W_4^+$ and $W_4^-$ are obtained from each other by reversing the orientations of all arcs, it suffices to prove Theorem~\ref{thm:dirdirac} for $\vec{K}_4, \vec{K}_4^s$ and $W_4^+$. While we can derive the result for the transitive tournament $\vec{K}_4$ directly from Theorem~\ref{subdivtrans4}, the proofs for $\vec{K}_4^s$ and $W_4^+$ are more involved and require some preparation.

\begin{proof}[Proof of $\text{mader}_{\vec{\chi}}(\vec{K}_4)=4$.]
Let $D$ be a given digraph with $\vec{\chi}(D) \ge 4$. Then $D$ contains a $4$-dicritical subdigraph $D' \subseteq D$. By Lemma~\ref{mindegree}, we have $\delta^+(D') \ge 3$. We now apply Theorem~\ref{subdivtrans4} to conclude that $D'$ and thus also $D$ contains a subdivision of $\vec{K}_4$. This completes the proof.
\end{proof}
\noindent
We prepare the proofs of $\text{mader}_{\vec{\chi}}(\vec{K}_4^s)=4$ and $\text{mader}_{\vec{\chi}}(W_4^+)=4$ with a set of useful lemmas.
\begin{lemma}\label{lem:butterfly1}
	Let $D$ be a digraph, let $(u,w) \in A(D)$, and let $D'$ be the digraph obtained from $D$ by deleting $u$ and adding the arc $(x,w)$ for each $x \in N_D^-(u) \setminus \{w\}$ (unless it already exists). Let $F$ be a sink-free orientation of a cubic graph. 
	If $D'$ contains a subdivision of $F$, then so does $D$.
\end{lemma}
\begin{proof}
Let $S' \subseteq D'$ be a subdigraph of $D'$ isomorphic to a subdivision of $F$. If $(x,w) \notin A(S')$ for all $x \in N_D^-(u)$, then $S'$ is also a subdigraph of $D$ and hence we have found a subdivision of $F$ in $D$. So suppose that $(x,w) \in A(S')$ for some $x \in N_D^-(u)$. We now distinguish between two \nolinebreak cases.

\textbf{Case 1:} There exists $x' \in N_D^-(u) \setminus \{x\}$ such that $(x',w) \in A(S')$. Then $w$ must be a branch vertex of the subdivision $S'$, and since $F$ is a sink-free orientation of a $3$-regular graph, there exists a unique third neighbor $x''$ of $w$ in $S'$ satisfying $(w,x'') \in A(S') \subseteq A(D')$. By definition of $D'$, we have $(w,x'') \in A(D)$ as well. We now see that the subdigraph $S$ of $D$ defined by $V(S):=V(S') \cup \{u\}$, $A(S):=(A(S') \setminus \{(x,w),(x',w)\}) \cup \{(x,u),(x',u),(u,w))\}$ forms a subdivision of $F$ in $D$, where the branch vertex $w$ of $S'$ is moved to the new branch vertex $u$ of $S$ (and $w$ becomes a subdivision vertex).

\textbf{Case 2:} $x$ is the unique vertex in $N_D^-(u)$ such that $(x,w) \in A(S')$. Then the subdigraph $S$ of $D$ defined by $V(S):=V(S') \cup \{u\}$ and $A(S):=(A(S') \setminus \{(x,w)\}) \cup \{(x,u),(u,w)\}$ forms a subdivision of $F$ contained in $D$.
\end{proof}
\begin{lemma}\label{cutvertices}
Let $D$ be a strongly connected digraph, let $v \in V(D)$, and let $(X,Y)$ be a non-trivial partition of $V(D)\setminus\{v\}$ such that $(x,y) \notin A(D)$ for all $x \in X, y \in Y$. Suppose further that $D[X]$ is strongly connected. Let $D_1:=D[X \cup \{v\}]$ and let $D_2$ be defined by $V(D_2):=Y \cup \{v\}$ and 
$A(D_2):=A(D[Y \cup \{v\}]) \cup \{(y,v) \; | \; y \in Y, x \in X, (y,x) \in A(D)\}$. Let further $F$ be a sink-free orientation of a cubic graph. Then
\begin{enumerate}
\item If $D_1$ or $D_2$ contains a subdivision of $F$, then so does $D$.
\item $\vec{\chi}(D) \le \max\{\vec{\chi}(D_1), \vec{\chi}(D_2)\}$.
\end{enumerate}
\end{lemma}
\begin{proof}
\noindent
\begin{enumerate}
\item The claim is trivial for $D_1$, since $D_1 \subseteq D$. Now suppose that $D_2$ contains a subdivision of $F$. The vertex $v$ in $D$ must have an in-neighbor in the set $X$, for otherwise $(X,Y \cup \{v\})$ would form a directed separation of $D$, contradicting the assumed strong connectivity. Since $D[X]$ is strongly connected, it follows that there exists an in-arborescence $T \subseteq D[X \cup \{v\}]$ rooted at $x_0:=v$ which spans $X \cup \{v\}$. Let $n:=|X|$, 
and fix an ordering $x_0,x_1,\ldots,x_n$ of the vertices of $T$ such that each vertex of $T$ appears before its children in the ordering (i.e., if $(x_j,x_i) \in A(T)$ then $j > i$). 
For every $i=n,n-1,n-2,\ldots,1,0$, let $H_i$ be the digraph obtained from $D$ by removing all arcs in $A(D[X \cup \{v\}]) \setminus A(T)$, deleting the vertices $\{x_{i+1},\ldots,x_n\}$ and adding the arc $(y,x_j)$ for every $y \in Y$ and  $j \in \{0,1,\ldots,i\}$ such that $y$ has an out-neighbor $x \in \{x_{i+1},\ldots,x_n\}$ in $D$ and the first intersection of the unique $x$-$x_0$-path in $T$ with $\{x_0,x_1,\ldots,x_i\}$ is $x_j$. Note that $H_n$ is a subdigraph of $D$ and that $H_0=D_2$. Further we can observe that for every $1 \le i \le n$, the digraph $H_{i-1}$ is obtained from $H_i$ by deleting $x_i$, and adding an arc from every $x \in N_{H_i}^-(x_i)$ to the parent of $x_i$ in $T$. Hence, repeated application of Lemma~\ref{lem:butterfly1} yields that if $D_2=H_0$ contains a subdivision of $F$, then the same is true for all $H_i, 0 \le i \le n$. Hence $H_n \subseteq D$ contains a subdivision of $F$, and this proves the claim.
\item Let $k:=\max\{\vec{\chi}(D_1),\vec{\chi}(D_2)\}$ and let $c_1:X \cup \{v\} \rightarrow \{1,\ldots,k\}$ and $c_2:Y \cup \{v\} \rightarrow \{1,\ldots,k\}$ be acyclic $k$-colorings of $D_1$ and $D_2$, respectively. Without loss of generality we may assume $c_1(v)=1=c_2(v)$. We now define a $k$-coloring of $V(D)$ by putting $c(x):=c_1(x)$ for every $x \in X$, $c(v):=1$, and $c(y):=c_2(y)$ for every $y \in Y$. We claim that this defines an acyclic $k$-coloring of $D$. Indeed, if not, then there exists a directed cycle $C$ in $D$ which is monochromatic under $c$. If $v \notin V(C)$, then since there is no arc from $X$ to $Y$ in $D$, we must have either $V(C) \subseteq X$ or $V(C) \subseteq Y$, which in both cases yields a contradiction to our choice of $c_1$ and $c_2$ as acyclic colorings. Hence, $v \in V(C)$ and $V(C) \subseteq c^{-1}(1)$. If $V(C) \cap Y=\emptyset$, then $C$ is a monochromatic cycle in the coloring $c_1$ of $D_1=D[X \cup \{v\}]$, a contradiction. We therefore have $v \in V(C), V(C) \cap Y \neq \emptyset$. Since there are no edges from $X$ to $Y$, there must be $w \in Y$ such that $(v,w) \in A(C)$. Let $C[v,w']$ be a maximal directed subpath of $C$ starting at $v$ such that $V(C[v,w'])\setminus \{v\} \subseteq Y$. (In other words, $C[v,w']$ is obtained by traversing $C$ starting from the arc $(v,w)$ and stopping just before the cycle leaves $Y$.) Then either $(w',v) \in A(D)$, or $(w',x) \in A(D)$ for some $x \in X$ and hence $(w',v) \in A(D_2)$ by definition of $D_2$. Therefore, $C[v,w']+(w',v)$ forms a directed cycle in $D_2$, all of whose vertices have color $1$ under $c_2$, contradicting our assumption on $c_2$. This contradiction shows that our initial assumption was wrong, namely that $c$ is indeed an acyclic coloring of $D$, proving that $\vec{\chi}(D) \le k=\max\{\vec{\chi}(D_1), \vec{\chi}(D_2)\}$. 
\end{enumerate}
\end{proof}
\begin{lemma}\label{lem:3intoone}
Let $D$ be a digraph, and let $u,v,w \in V(D)$ be pairwise distinct such that $v,w \in N^+(u) \cap N^-(u)$ (i.e., $\{u,v\}$ and $\{u,w\}$ induce digons).  Let $D^\ast$ be obtained from $D$ by deleting $v$ and $w$ and adding the arcs 
$$\{(u,x) \; | \; x \in V(D)\setminus \{u,v,w\}, N^-(x) \cap \{v,w\} \neq \emptyset\}$$
and 
$$\{(x,u) \; | \; x \in V(D)\setminus \{u,v,w\}, N^+(x) \cap \{v,w\} \neq \emptyset\}.$$ 
Let $F$ be an orientation of a cubic graph. 
If $D^\ast$ contains a subdivision of $F$, then so does $D$.
\end{lemma}
\begin{proof}
Let $S^\ast$ be a subdigraph of $D^\ast$ isomorphic to a subdivision of $F$. If $u \notin V(S^\ast)$, then $S^\ast$ is also a subdigraph of $D$ and we are done. Hence, suppose in the following that $u \in V(S^\ast)$. If $u$ is a subdivision vertex in $S^\ast$, then let $u^- \in N_{D^\ast}^-(u)$ and $u^+ \in N_{D^\ast}^+(u)$ denote the in- and the out-neighbor of $u$ in $S^\ast$, respectively. By definition of $D^\ast$ there exist $x^-,x^+ \in \{u,v,w\}$ such that $(u^-,x^-), (x^+,u^+) \in A(D)$. Let $P$ denote the bioriented path with vertex-trace $v,u,w$. Then clearly $P$ contains a directed $x^-,x^+$-path $P_{x^-,x^+}$. Now 
$A(S):=(A(S^\ast)\setminus\{(u^-,u),(u,u^+)\}) \cup \{(u^-,x^-),(x^+,u^+)\} \cup A(P_{x^-,x^+})$ forms the arc-set of a subdigraph $S \subseteq D$ isomorphic to a subdivision of $F$. For the next case suppose that $u$ is a branch vertex of the subdivision $S^\ast$. For every in-neighbor $y \in N_{S^\ast}^-(u)$ in $S^\ast$, let $x(y) \in \{u,v,w\}$ be a vertex such that $(y,x(y)) \in A(D)$, and for every out-neighbor $y \in N_{S^\ast}^+(u)$, let $x(y) \in \{u,v,w\}$ be a vertex such that $(x(y),y) \in A(D)$. Let $y_1,y_2,y_3$ be the three distinct neighbors of $u$ in $S^\ast$, ordered in such a way that $x(y_2)$ lies on the unique bioriented subpath $P_{x(y_1),x(y_3)}$ of $P$ connecting the vertices $x(y_1)$ and $x(y_3)$. It is now evident that the subdigraph of $D$ obtained from $S^\ast$ by deleting $u$ and adding $P_{x(y_1),x(y_3)}$ and the arcs $(y_i,x(y_i))$ for $y_i \in N_{S^\ast}^-(u)$ and $(x(y_i),y_i)$ for $y_i \in N_{S^\ast}^+(u)$, contains a subdivision of $F$ with $x(y_2)$ as a branch vertex. This verifies the claim in the second case as well and concludes the proof. 
\end{proof}
\begin{proof}[Proof of $\text{mader}_{\vec{\chi}}(\vec{K}_4^s)=4$]
Suppose towards a contradiction that the claim is wrong, and let $D$ be a counterexample minimizing lexicographically the pair $(|V(D)|,|A(D)|)$; namely, the number of vertices is minimized with first priority and the number of arcs with second priority. Clearly, $|V(D)| \ge 5$, $D$ is $4$-dicritical, and it contains no subdivision of $\vec{K}_4^s$. By Lemma \ref{strong connectivity}, $D$ is strongly-connected.
\paragraph{Claim 1.} $D$ is strongly $2$-vertex-connected.
\begin{proof}
Suppose towards a contradiction that there exists a vertex $v \in V(D)$ such that $D-v$ is not strongly connected. This means that $D-v$ has more than one strong component. Let $X \subseteq V(D-v)$ be the vertex set of a strong component of $D-v$ which is a ``sink'' in $D-v$, that is, there is no arc leaving $X$. Let $Y:=V(D) \setminus (X \cup \{v\})$. Then $(X,Y)$ forms a partition of $V(D)\setminus\{v\}$, $D[X]$ is strongly connected and $(x,y) \notin A(D)$ for all $x \in X, y \in Y$. We can therefore apply Lemma~\ref{cutvertices} with $F=\vec{K}_4^s$ to obtain a pair $D_1,D_2$ of digraphs with vertex-sets $X \cup \{v\}, Y \cup \{v\}$, respectively, such that neither $D_1$ nor $D_2$ contains a subdivision of $\vec{K}_4^s$ and $4=\vec{\chi}(D) \le \max\{\vec{\chi}(D_1),\vec{\chi}(D_2)\}$. However, this means that there is some $i \in \{1,2\}$ such that $\vec{\chi}(D_i) \ge 4$, $D_i$ contains no $\vec{K}_4^s$-subdivision and clearly $|V(D_i)|<|V(D)|$. This contradicts the assumed minimality of $D$, thus showing that the assumption was wrong, namely that $D$ is indeed strongly $2$-vertex-connected.
\end{proof}
\paragraph{Claim 2.} $\delta^+(D), \delta^-(D) \ge 4$.
\begin{proof}
Note that $\vec{K}_4^s$ is isomorphic to the tournament obtained from it by reversing all arcs. 
It follows that since $D$ is a counterexample to the claim, so is $\revvec{D}$, which is the digraph obtained from $D$ by reversing all its arcs. Evidently, we have $(|V(\revvec{D})|,|A(\revvec{D})|) = (|V(D)|,|A(D)|)$, meaning that $\revvec{D}$ is also a minimal counterexample (in the sense defined above). Since $\delta^-(D)=\delta^+(\revvec{D})$, it suffices to prove $\delta^+(D) \ge 4$.

Suppose towards a contradiction that there exists a vertex $u \in V(D)$ such that $d^+(u) \le 3$. Since $D$ is $4$-dicritical, Lemma~\ref{mindegree} implies that $d^+(u)=3$. We now distinguish between two cases depending on the structure of the out-neighborhood of $u$.

\textbf{Case 1:} There exists some $w \in N^+(u)$ such that $(w,u) \notin A(D)$. 
In this case, let $D'$ be the digraph defined as in Lemma~\ref{lem:butterfly1}. Namely, $D'$ is obtained from $D$ by deleting $u$ and adding the arcs $(x,w)$ for all $x \in N^-(u)$. By Lemma~\ref{lem:butterfly1}, $D'$ contains no subdivision of $\vec{K}_4^s$. Since $|V(D')|=|V(D)|-1$, the minimality assumption on $D$ implies that $\vec{\chi}(D') \le 3$. So let $c':V(D) \setminus \{u\} \rightarrow \{1,2,3\}$ be an acyclic $3$-coloring of $D'$. Write $N_D^+(u)=\{w,w_1,w_2\}$. Fix a color $c_u \in \{1,2,3\}\setminus\{c'(w_1),c'(w_2)\}$ (which clearly exists). Let $c:V(D) \rightarrow \{1,2,3\}$ be the coloring of $D$ defined by $c(x):=c'(x)$ for all $x \in V(D) \setminus \{u\}$ and $c(u):=c_u$. Since $\vec{\chi}(D)=4$, there has to be a directed cycle $C$ in $D$ which is monochromatic under $c$. Clearly, $C$ has to pass through $u$, for otherwise it would have been a monochromatic dicycle already in the coloring $c'$ of $D'$. Since none of the out-arcs $(u,w_1), (u,w_2)$ is monochromatic, we must have $(u,w) \in E(C)$. Let $u' \in N_D^-(u)$ be the unique predecessor of $u$ on $C$. Then  $u' \neq w$ because $(w,u) \notin A(D)$ by assumption. It follows from the definition of $D'$ that replacing the directed subpath $u',(u',u),u,(u,w),w$ of $C$ with the (``direct") arc $(u',w)$ in $D'$ defines a directed cycle $C'$ in $D'$ such that $V(C')=V(C)\setminus \{u\}$. Hence, $C'$ is a monochromatic dicycle in the acyclic coloring $c'$ of $D'$. This contradiction shows that our initial assumption $d^+(u) \le 3$ was wrong.

\textbf{Case 2:} $(w,u) \in A(D)$ for all $w \in N^+(u)$. We claim that in this case, we can find a pair $w_1, w_2 \in N^+(u)$ of distinct neighbors of $u$ such that $(w_1,w_2) \notin A(D)$. Indeed, suppose this were not the case. Then the vertices $\{u\} \cup N^+(u)$ induce a $\bivec{K}_4$ in $D$. However, this clearly means that $D$ contains $\vec{K}_4^s \subseteq \bivec{K}_4$ as a subdigraph, contradicting our initial assumption on $D$. So let us fix, in the following, a pair of distinct $w_1, w_2 \in N^+(u) \subseteq N^-(u)$ such that $(w_1,w_2) \notin A(D)$. Let $D^\ast$ be the digraph obtained from $D$ by 
applying the operation of Lemma~\ref{lem:3intoone} to $\{u,w_1,w_2\}$; that is, we delete $w_1$ and $w_2$ and 
add the arc $(u,x)$ for every $x \in V(D) \setminus \{u,w_1,w_2\}$ which has an in-neighbor in $\{w_1,w_2\}$ and the arc $(x,u)$ for every $x \in V(D) \setminus \{u,w_1,w_2\}$ which has an out-neighbor in $\{w_1,w_2\}$.
By Lemma~\ref{lem:3intoone}, $D^\ast$ does not contain a subdivision of $\vec{K}_4^s$. We clearly have $|V(D^\ast)|<|V(D)|$ and so the minimality assumption on $D$ yields that there is an acyclic $3$-coloring $c^\ast:V(D^\ast) \rightarrow \{1,2,3\}$ of $D^\ast$. Write $N^+(u)=\{w_1,w_2,w_3\}$, and let $c_u \in \{1,2,3\}$ be a color distinct from both $c^\ast(u)$ and $c^\ast(w_3)$. We now define a $3$-coloring $c$ of $D$ by putting $c(x):=c^\ast(x)$ for all $x \in V(D)\setminus \{u,w_1,w_2\}$, $c(u):=c_u$, and $c(w_1):=c(w_2):=c^\ast(u)$. Since $\vec{\chi}(D)=4$, there must be a dicycle $C$ in $D$ which is monochromatic under $c$. Then $C$ cannot contain $u$, for otherwise it would have to leave $u$ through one of the out-arcs $(u,w_1),(u,w_2),(u,w_3)$, but by the definition of the coloring $c$, none of these arcs is monochromatic. On the other hand, we must have $V(C) \cap \{w_1,w_2\}\neq\emptyset$, for otherwise $C$ would be a monochromatic dicycle in $(D^\ast,c^\ast)$, which is impossible. Observe also that $V(C) \setminus \{w_1,w_2\} \neq \emptyset$ because $(w_1,w_2) \notin A(D)$. Let $x_0,x_1,\ldots,x_\ell=x_0$ be the vertex-trace of $C$ in $D$. Now consider the closed sequence $y_0,y_1,\ldots,y_{\ell}=y_0$ of vertices in $D^\ast$, where $y_i:=x_i$ if $x_i \notin \{w_1,w_2\}$ and $y_i:=u$ if $y_i \in \{w_1,w_2\}$. The definitions of $D^\ast$ and $c$ and the fact that $u \notin V(C)$ imply that $c^\ast(y_i)=c(x_i)$ for every $i=1,\ldots,\ell$, and that either $(y_{i-1},y_i) \in A(D^\ast)$ or $y_{i-1}=y_i=u$ for every $i=1,\ldots,\ell$. This means that in $D^\ast$ there is a monochromatic closed directed walk which contains at least two vertices: it contains $u$ because $V(C) \cap \{w_1,w_2\} \neq \emptyset$ and at least one other vertex because $V(C) \setminus \{w_1,w_2\} \neq \emptyset$ and $u \notin V(C)$. Therefore, $D^\ast$ contains a monochromatic dicycle. All in all, this contradicts the fact that $c^\ast$ was chosen as an acyclic coloring of $D^\ast$, implying that our initial assumption $d^+(u) \le 3$ was wrong.

To sum up, we have arrived at a contradiction in both cases, which means that we indeed must have $\delta^+(D) \ge 4$. As argued above, we can derive $\delta^-(D)=\delta^+(\revvec{D}) \ge 4$ with the same arguments applied to the minimal counterexample $\revvec{D}$. This finishes the proof of the claim.
\end{proof}
With Claims 1 and 2 at hand, we can now apply Theorem~\ref{critical strongly connected} to $D$ with $k=2$, and thus obtain a vertex $v \in V(D)$ such that $D-v$ is strongly $2$-vertex-connected. We now complete the proof of the Theorem by explicitly constructing a subdivision of $\vec{K}_4^s$ in $D$. We start with the following observation.
\paragraph{Claim 3.} There are $3$ directed cycles $C_1$, $C_2$, $C_3$ in $D$ such that $V(C_i) \cap V(C_j)=\{v\}$ for any two distinct $i, j \in \{1,2,3\}$.
\begin{proof}
Since $D$ is $4$-dicritical, $D-v$ admits an acyclic coloring with colors $\{1,2,3\}$.
For every $i \in \{1,2,3\}$, if we try and extend this coloring to $D$ by assigning color $i$ to $v$, we have to find a monochromatic directed cycle $C_i$ in $D$, which has to pass through $v$. Note that $V(C_i) \cap V(C_j) = \{v\}$ for all $1 \leq i < j \leq 3$, because all vertices in $V(C_i) \setminus \{v\}$ receive color $i$ ($1 \leq i \leq 3$).
\end{proof}
\noindent
The rest of the proof is divided into two cases depending on the lengths of the cycles $C_i$. 

\textbf{Case 1.} All the three cycles $C_1, C_2, C_3$ have length two, i.e., are digons.
Let $v_1, v_2, v_3 \in V(D)$ be such that $V(C_i)=\{v,v_i\}$ ($i = 1,2,3$). Since $D-v$ is strongly connected, there has to be a directed path in $D-v$ starting in $v_1$ and ending in $\{v_2,v_3\}$. Let $P$ be a shortest such directed path, and without loss of generality assume that it ends in $v_2$. By the minimality assumption on $P$ we know that $v_3 \notin V(P)$. Now put $A:=V(P), B:=N_{D-v}^-(v_3) \subseteq V(D-v)$. We clearly have $|A| \ge |\{v_1,v_2\}|=2, |B|=d^-(v_3)-1 > 2$, and hence we may apply Theorem~\ref{setmenger} to $D-v$ and obtain that there are two vertex-disjoint $A$-$B$-dipaths $P_1'$ and $P_2'$ in $D-v$. We may assume that $v_3 \notin V(P_i')$ for $i=1,2$ (otherwise we can simply delete $v_3$ and all its successors from $P_i'$). 
For $i = 1,2$, let $P_i := P'_i + v_3$, and write $V(P) \cap V(P'_i)=:\{w_i\}$.
Then $P_1$ and $P_2$ only intersect at $v_3$, and $P$ only intersects $P_i$ at $w_i$ (for $i = 1,2$). 
Without loss of generality (by relabeling if necessary), we may assume that when traversing $P$ from $v_1$ towards $v_2$, we first meet $w_1$ before we meet $w_2$. Now let $S$ be the subdigraph of $D$ defined by the union of the following dipaths in $D$: $P$, $P_1$, $P_2$, $(v, (v,v_1), v_1)$, $(v_3, (v_3,v), v)$ and $(v_2, (v_2,v),v)$. It is now easy to observe that $S$ constitutes a subdivision of $\vec{K}_4^s$ whose branch vertices are $v, w_1, w_2, v_3$. This contradicts our initial assumption that $D$ contains no subdivision of $\vec{K}_4^s$. 

\textbf{Case 2.} There is some $i \in \{1,2,3\}$ such that $|V(C_i)| \ge 3$.
Without loss of generality we may assume that $i=1$. Let $v_2$ be the unique out-neighbor of $v$ on $C_2$. Put $A:=N_{D-v}^+(v_2), B:=V(C_1) \setminus \{v\} \subseteq V(D-v)$. Clearly, $|A| \ge 4-1>2, |B| \ge 3-1=2$, and hence we may apply Theorem~\ref{setmenger} to $D-v$ to conclude that there are two vertex-disjoint $A$-$B$-dipaths $P_1'$, $P_2'$ in $D-v$. We may assume that $v_2 \notin V(P_i')$ for $i=1,2$ (otherwise we can simply delete $v_2$ and all its predecessors from $P_i'$). 
For $i=1,2$, let $x_i \in A$ and $y_i \in B$ be the endpoints of $P'_i$. 
It is now clear that the union of $C_1$ and the internally vertex-disjoint dipaths $Q:=(v,(v,v_2),v_2)$, $P_1:=(v_2,(v_2,x_1),P_1')$ and $P_2:=(v_2,(v_2,x_2),P_2')$ is a subdivision of $\vec{K}_4^s$ in $D$ with branch vertices $v,v_2,y_1,y_2$. 
This contradicts our initial assumption that $D$ contains no subdivision of $\vec{K}_4^s$. 

Since we arrived at a contradiction in both cases, it follows that our initial assumption that there exists a (smallest) digraph $D$ with $\vec{\chi}(D) \ge 4$ not containing a subdivision of $\vec{K}_4^s$ was wrong. This finishes the proof.
\end{proof}

We now move on to show that $\text{mader}_{\vec{\chi}}(W_4^+)=4$. This proof is partly inspired by a method used in \cite{HMM}. 

\begin{proof}[Proof of $\text{mader}_{\vec{\chi}}(W_4^+)=4$.]
Suppose towards a contradiction that there exists a digraph $D$ such that $\vec{\chi}(D) \ge 4$, but $D$ contains no subdivision of $W_4^+$. Assume additionally that $D$ lexicographically minimizes the pair $(|V(D)|,|A(D)|)$ (i.e., the number of vertices is minimized with first priority, and the number of arcs is minimized with second priority). Clearly, $|V(D)| \ge 5$ and $D$ is $4$-dicritical. Hence,
 $\delta^+(D), \delta^-(D) \ge 3$ by Lemma~\ref{mindegree}, and $D$ is strongly-connected by Lemma~\ref{strong connectivity}. 
\paragraph{Claim 1.} $D$ is strongly $2$-vertex-connected.
\begin{proof}
Suppose towards a contradiction that there exists a vertex $v \in V(D)$ such that $D-v$ is not strongly connected. This means that $D-v$ has more than one strong component. Let $X \subseteq V(D-v)$ be the vertex set of a strong component of $D-v$ which is a ``sink'' in $D-v$, that is, there is no arc leaving $X$. Let $Y:=V(D) \setminus (X \cup \{v\})$. Then $(X,Y)$ forms a partition of $V(D)\setminus\{v\}$, $D[X]$ is strongly connected and $(x,y) \notin A(D)$ for all $x \in X, y \in Y$. We can therefore apply Lemma~\ref{cutvertices} with $F=W_4^+$ to obtain a pair $D_1,D_2$ of digraphs with vertex-sets $X \cup \{v\}, Y \cup \{v\}$, respectively, such that neither $D_1$ nor $D_2$ contains a subdivision of $W_4^+$ and $4=\vec{\chi}(D) \le \max\{\vec{\chi}(D_1),\vec{\chi}(D_2)\}$. However, this means that there is some $i \in \{1,2\}$ such that $\vec{\chi}(D_i) \ge 4$, $D_i$ contains no $W_4^+$-subdivision and clearly $|V(D_i)|<|V(D)|$. This contradicts the assumed minimality of $D$. This contradiction shows that the assumption was wrong, namely that $D$ is indeed strongly $2$-vertex-connected.
\end{proof}
\paragraph{Claim 2.} The underlying graph of $D$ is $3$-vertex-connected.
\begin{proof}
Suppose towards a contradiction that there is a set $K \subseteq V(D)$ such that $|K| \le 2$ and $D-K$ is not weakly connected. Let $(X_1,X_2)$ be a partition of $V(D)\setminus K$ into two non-empty sets such that there is no arc between $X_1$ and $X_2$ in $D-K$. Since $D$ is strongly $2$-vertex-connected, we must have $|K|=2$, say $K = \{s_1,s_2\}$ for some distinct $s_1,s_2 \in V(D)$. 
For $i = 1,2$, let $D_i$ be the digraph defined as follows:
$V(D_i):=X_i \cup K$ and $A(D_i):=A(D[X_i \cup K]) \cup \{(s_1,s_2),(s_2,s_1)\}$. We claim that none of $D_1, D_2$ contains a subdivision of $W_4^+$. Indeed, suppose towards a contradiction that for some $i \in \{1,2\}$, there exists a subdigraph $S \subseteq D_i$ which is isomorphic to a subdivision of $W_4^+$. If $A(S) \cap \{(s_1,s_2),(s_2,s_1)\}=\emptyset$, then $S$ would also be a subgraph of $D$, contradicting our assumptions on $D$. Hence, $S$ has to contain an arc between $s_1$ and $s_2$, and without loss of generality we may assume that $(s_1,s_2) \in A(S)$. Since $W_4^+$ contains no digons, the same is true for $S$ and hence $(s_2,s_1) \notin A(S)$. We now claim that there exists an $s_1$-$s_2$-dipath in $D-X_i$. To this end, choose an arbitrary vertex $w \in X_{3-i}$. Since both $D-s_1$ and $D-s_2$ are strongly connected (by Claim 1), there are dipaths $P_1$ and $P_2$ in $D-s_1$ resp. $D-s_2$ such that $P_1$ starts at $w$ and ends at $s_2$, while $P_2$ starts at $s_1$ and ends at $w$. 
Since $D$ contains no arcs between $X_1$ and $X_2$, we must have $V(P_1) \subseteq X_{3-i} \cup \{s_2\}$ and $V(P_2) \subseteq X_{3-i} \cup \{s_1\}$. Finally, we see that the concatenation of $P_1$ and $P_2$ is a directed walk from $s_1$ to $s_2$, implying that $P_1 \cup P_2 \subseteq X_{3-i} \cup \{s_1,s_2\} = D-X_i$ contains an $s_1$-$s_2$-dipath $P$. 
Clearly, $P$ is internally vertex-disjoint from all dipaths in $S-(s_1,s_2) \subseteq D$, and hence the subdigraph $S':=(S-(s_1,s_2)) \cup P$ of $D$ is isomorphic to a subdivision of $S$, which in turn is a subdivision of $W_4^+$. This contradicts our initial assumption that $D$ is $W_4^+$-subdivision-free. We conclude that neither $D_1$ nor $D_2$ contains a subdivision of $W_4^+$, as claimed. Since clearly $|V(D_1)|, |V(D_2)|<|V(D)|$, the assumed minimality of $D$ yields that $D_1$ and $D_2$ admit acyclic $3$-colorings $c_1:V(D_1) \rightarrow \{1,2,3\}$ and $c_2:V(D_2) \rightarrow \{1,2,3\}$, respectively. Since the pair $s_1,s_2$ induces a digon in both $D_1$ and $D_2$, we must have $c_1(s_1) \neq c_1(s_2), c_2(s_1) \neq c_2(s_2)$. Hence, possibly after permuting the color set, we may assume that $c_1(s_1)=c_2(s_1)=1, c_1(s_2)=c_2(s_2)=2$. We now claim that the common extension $c:V(D) \rightarrow \{1,2,3\}$ of $c_1$ and $c_2$ to $D$ defines an acyclic coloring of $D$. Indeed, a monochromatic directed cycle $C$ in $(D,c)$ would have to contain vertices of both $X_1$ and $X_2$, for otherwise it would also be a monochromatic dicycle in $(D_1,c_1)$ or $(D_2,c_2)$, contradicting the assumption that these are acyclic colorings. However, since $K = \{s_1,s_2\}$ separates $X_1$ and $X_2$, this is only possible if $\{s_1,s_2\} \subseteq V(C)$. But then $C$ is not monochromatic because $c(s_1) = 1$ and $c(s_2) = 2$. This shows that $c$ is indeed an acyclic coloring of $D$, which in turn contradicts $\vec{\chi}(D)=4$. So we see that our initial assumption that the underlying graph of $D$ admits a $2$-separator was wrong. This concludes the proof of Claim 2.
\end{proof}
\paragraph{Claim 3.} For every $x \in V(D)$ there is a directed cycle $C$ in $D-x$ such that $|V(C)| \ge 3$.
\begin{proof}
Let $x \in V(D)$ be given arbitrarily. Suppose towards a contradiction that every directed cycle in the digraph $D-x$ has length two, i.e., is a digon. Recall that in a strongly connected digraph, every arc lies on a directed cycle. Since $D-x$ is strongly connected (Claim 1), every arc of $D-x$ is contained in a digon, and hence $D-x$ is a symmetric digraph. Since $D - x$ contains no directed cycle of length at least $3$, this is only possible if $D-x$ is a biorientation of a forest. But then the bipartition of this forest defines an acyclic $2$-coloring of $D-x$. By assigning to $x$ a distinct third color, we obtain an acyclic $3$-coloring of $D$, a contradiction to $\vec{\chi}(D)=4$. This proves the claim.
\end{proof}
\paragraph{Claim 4.} For every pair $(x,C)$ of a vertex in $D$ and a directed cycle $C$ in $D-x$ of length at least $3$, there exists a partition $(W,K,Z)$ of $V(D)$ with the following properties: 
\begin{itemize}
\item $x \in W$ and $V(C) \subseteq K \cup Z$
\item There is no arc in $D$ with tail in $W$ and head in $Z$.
\item $|K|=2$.
\end{itemize}
A partition $(W,K,Z)$ with these properties will be called a \emph{good separation} for the pair $(x,C)$.
\begin{proof}
 We claim that there are no three $x$-$V(C)$-dipaths in $D$ which pairwise intersect only at $x$. Indeed, three such dipaths joined with $C$ would form a subdivision of $W_4^+$, which does not exist in $D$ by assumption. By Theorem~\ref{setmenger}, there is a set $K \subseteq V(D)\setminus \{x\}$ of size at most $2$ such that there are no $x$-$V(C)$-dipaths in $D-K$. Let $W \subseteq V(D) \setminus K$ be the set of vertices  reachable in $D - K$ via a dipath starting at $x$, and let $Z:=V(D)\setminus (W \cup K)$. It follows now directly by definition that $x \in W$ and $V(C) \subseteq K \cup Z$, and there is no arc with tail in $W$ and head in $Z$. We have $|K| \le 2$, and since $D$ is strongly $2$-vertex-connected (by Claim 1), it follows that $|K|=2$. Therefore $(W,K,Z)$ is a good separation of the pair $(x,C)$. 
\end{proof}
In the following, for every pair $(x,C)$ of a vertex $x \in V(D)$ and a directed cycle $C$ in $D-x$ of length at least $3$, we denote by $\omega(x,C)$ the minimum of $|W|$ over all good separations $(W,K,Z)$ of $(x,C)$. Let $\omega_0:=\min\{\omega(x,C) \; | \; x \in V(D), C \text{ dicycle in }D-x, |V(C)| \ge 3\}$.

\paragraph{Claim 5.} Let $x \in V(D)$, let $C$ be a directed cycle in $D-x$ of length at least $3$, and let $(W,K,Z)$ be a good separation for $(x,C)$ such that $|W| = \omega(x,C)$. Then every vertex of $W$ is reachable from $x$ by a dipath in $D[W]$. 
\begin{proof}
	Let $W'$ be the set of all vertices $w \in W$ which are reachable from $x$ in $D[W]$. Evidently, $x \in W' \subseteq W$. Observe that $(W',K,(W \setminus W') \cup Z)$ forms a good separation for $(x,C)$, because $D$ has no edge with tail in $W'$ and head in $W \setminus W'$. It follows that $|W'| \geq \omega(x,C) = |W|$. This implies that $W' = W$, as required.  
\end{proof}

\paragraph{Claim 6.} There exists a pair $(x^\ast,C^\ast)$ of a vertex $x^\ast \in V(D)$, a dicycle $C^\ast$ of length at least $3$ in $D-x^\ast$, and there exists a good separation $(W,K,Z)$ of $(x^\ast,C^\ast)$, such that the following hold:
\begin{itemize}
\item $|W|=\omega_0$,
\item there exist $z^\ast \in Z, w^\ast \in W \setminus \{x^\ast\}$ such that $(z^\ast,w^\ast) \in A(D)$.
\end{itemize}
\begin{proof}
Let $(x_0,C_0)$ be a pair of a vertex and a disjoint dicycle in $D$, such that $|V(C_0)| \ge 3$, and such that $(x_0,C_0)$ attains the minimum $\omega_0=\omega(x_0,C_0)$. Let $(W,K,Z)$ be a good separation for $(x_0,C_0)$ such that $|W|=\omega(x_0,C_0)=\omega_0$. Note that $(W,K,Z)$ is also a good separation for every pair $(x,C_0)$ where $x \in W \setminus \{x_0\}$. 


Observe that there has to exist an arc between $Z$ and $W$, for if not, then $D-K$ is not weakly connected, contradicting the facts that $|K|=2$ and that the underlying graph of $D$ is $3$-vertex-connected (Claim 2). As there are no arcs from $W$ to $Z$, there has to exist an arc from $Z$ to $W$. Let $(z_0,w_0)$ be such an arc. If $w_0 \neq x_0$, then we directly obtain that $(x_0,C_0)$ together with $(W,K,Z)$ and the arc $(z_0,w_0)$ satisfy all the required properties in the statement of the claim. If $x_0=w_0$, then choose some $x_1 \in W \setminus \{x_0\}$. Such a selection is possible, since $ N^+(x_0) \setminus K \subseteq W$ and $N^+(x_0) \setminus K \neq \emptyset$ as $d^+(x_0) \ge \delta^+(D) \ge 3$. Since $(W,K,Z)$ is a good separation also for $(x_1,C_0)$, and since $x_1 \neq w_0 = x_0$, it follows now that $(x_1,C_0)$ together with $(W,K,Z)$ and the arc $(z_0,w_0)$ have all the claimed properties. This concludes the proof of Claim 6.
\end{proof}
In the following, let us consider a pair $(x^\ast, C^\ast)$ together with the good separation $(W,K,Z)$ and the arc $(z^\ast,w^\ast)$ as given by Claim 6. Since $D-x^\ast$ is strongly connected (Claim 1), and since $(z^\ast,w^\ast) \in A(D-x^\ast)$, there exists a directed cycle $C'$ in $D-x^\ast$ passing through $(z^\ast,w^\ast)$. As $D$ has no arc from $W$ to $Z$, the dicycle $C'$ must use at least one vertex from $K$, which means that $|V(C')| \ge 3$. Write $K=\{s_1,s_2\}$ and assume (without loss of generality) that $s_1 \in V(C')$. 

By Claim~4, there exists a good separation $(W',K',Z')$ for the pair $(x^\ast,C')$; choose it such that $|W'|=\omega(x^\ast,C')$, and such that it minimizes $|K' \cap Z|$ among all such good separations. We claim that $K' \cap W \neq \emptyset$. Indeed, if we had $K' \cap W = \emptyset$ then we would have $D[W] \subseteq D - K'$, which would imply that $w^\ast \in W$ is reachable from $x^\ast$ by a directed path in $D-K'$ (as every vertex of $W$ is reachable from $x^\ast$ by a dipath in $D[W]$ by Claim 5). However, this would contradict the facts that $x^\ast \in W'$, $w^\ast \in V(C') \subseteq Z'$, and there is no path from $W'$ to $Z'$ in $D - K'$ (by the definition of a good separation). Let us write $K'=\{s_1',s_2'\}$, where $s_1' \in W$.

\paragraph{Claim 7.} $K' \cap Z = \emptyset$.
\begin{proof}


Suppose towards a contradiction that $K' \cap Z \neq \emptyset$,
which means that $s_2' \in Z$ (because $s'_1 \in W$). Let $R$ be the set of vertices  reachable from $x^\ast$ via a dipath in the digraph $D-\{s_1',s_2\}$. We claim that $R \subseteq W'$. Suppose towards a contradiction that $R\setminus W' \neq \emptyset$. Then there is an $(x^\ast,R \setminus W')$-dipath $P$ in $D-\{s_1',s_2\}$. Note that $V(P) \subseteq R$ by the definition of $R$. Let $y \in R \setminus W'$ be the end-vertex of $P$ and $y'$ its predecessor. Then $y' \in W'$ because only the last vertex $y$ of $P$ is in $R \setminus W'$.
Since $(y',y)$ cannot have its tail in $W'$ and head in $Z'$, we must have $y \in (V(D) \setminus \{s_1',s_2\}) \setminus (W' \cup Z') \subseteq \{s_2'\}$; hence $y=s_2' \in Z$. Since $x^\ast \in W$, and since $K$ separates $W$ from $Z$, there must be a vertex on $P-y$ which belongs to $K=\{s_1,s_2\}$. However, this vertex can be neither $s_1$ nor $s_2$; indeed, $s_1 \notin V(P) \setminus \{y\}$ because $V(P) \setminus \{y\} \subseteq W'$ and $s_1 \in V(C') \subseteq K' \cup Z'$, and $s_2 \notin V(P)$ because $P$ is contained in $D-\{s_1',s_2\}$. This contradiction shows that $R \subseteq W'$, as claimed. 

There is no arc from $R$ to $V(D) \setminus (R \cup \{s_1',s_2\}))$ by the definition of $R$, which means that $|\{s_1',s_2\}|=2$ since $D$ is strongly $2$-vertex-connected (by Claim 1). We furthermore have $x^\ast \in R$ and $V(C') \subseteq V(D) \setminus W' \subseteq V(D) \setminus R$. Hence, $(R,\{s_1',s_2\},V(D) \setminus (R \cup \{s_1',s_2\}))$ defines a good separation for the pair $(x^\ast,C')$. By the choice of $(W',K',Z')$, this means that $R=W'$ and $|R|=\omega(x^\ast,C')$. We further have $|\{s_1',s_2\} \cap Z|= 0 < 1 = |K' \cap Z|$ (because $s_1' \in W$ and $s_2 \in K$), contradicting our choice of $(W',K',Z')$. This contradiction shows that the assumption $K' \cap Z \neq \emptyset$ was wrong, concluding the proof of the claim.
\end{proof}

\paragraph{Claim 8.} $W' \cap Z \neq \emptyset$ and $Z' \cap Z \neq \emptyset$.
\begin{proof}
We start by showing that $W' \cap Z \neq \emptyset$. 
Suppose towards a contradiction that $W' \cap Z=\emptyset$. Since $\{ w^*,s_1 \} \subseteq V(C') \subseteq K' \cup Z'$, we have $W' \cap \{ w^*,s_1 \} = \emptyset$ and hence
$W' \subseteq V(D) \setminus (Z \cup \{w^\ast,s_1\}) = (W \cup K) \setminus \{w^\ast, s_1\}$ $=(W \setminus \{w^\ast\}) \cup \{s_2\}$. On the other hand, we have $|W'| = \omega(x^\ast,C') \ge \omega_0=|W|$, and hence $W'=(W \setminus \{w^\ast\}) \cup \{s_2\}$. In particular, $s_2 \in W'$. Since $D$ is strongly $2$-vertex-connected, $s_2$ must have an out-neighbor $y \in Z$, for otherwise there would be no directed path from $x^\ast$ to $Z$ in $D-s_1$. Using Claim 7 and our assumption that $W' \cap Z = \emptyset$, we have $(W' \cup K') \cap Z=\emptyset$, and hence 
$y \in V(D) \setminus (W' \cup K') = Z'$. However, this means that $(s_2,y)$ is an arc from a vertex in $W'$ to a vertex in $Z'$, a contradiction. 
This contradiction shows that the initial assumption $W' \cap Z = \emptyset$ was wrong, proving the first part of claim.

For the second part, recall that $z^\ast \in Z$ and $z^\ast \in V(C') \subseteq K' \cup Z'$. Since $K' \cap Z=\emptyset$ (by Claim 7), we conclude that $z^\ast \in Z' \cap Z$ and hence $Z' \cap Z \neq \emptyset$, as required. 	
\end{proof}


\paragraph{Claim 9.} Every dipath in $D$ starting in $W' \cap Z \neq \emptyset$ and ending in $Z' \cap Z \neq \emptyset$ must contain $s_1$.
\begin{proof}
We first establish that $s_2 \in W'$. To see this, pick some vertex $v \in W' \cap Z$. By Claim 5 and as $|W'| = \omega(x^\ast,C')$, there exists an $x^\ast$-$v$-dipath in $D[W']$. Since $x^\ast \in W$ and $v \in Z$, this dipath must contain a vertex from $K$. However, since $s_1 \notin W'$, this vertex must be $s_2$, implying that $s_2 \in W'$.

Now to prove the claim, let $P$ be a directed path starting in a vertex $a \in W' \cap Z$ and ending in a vertex $b \in Z' \cap Z$. Let $y \in V(P)$ be the last vertex on $P$ contained in $W' \cup K'$ when traversing $P$ starting from $a$. Let $y'$ be the successor of $y$ on $P$; then $(y,y') \in A(D)$ and $y' \in Z'$. Hence, we must have $y \in K'$ (since $D$ has no arcs from $W'$ to $Z'$). It now follows from Claim 7 that 
$y \in W \cup K$. Now let us consider the subpath $P[y,b]$ starting at $y$ and ending at $b$. By definition of $y$, no vertex on $P[y,b]$ is contained in $W'$, and hence $s_2$ does not lie on this path. However, $P[y,b]$ starts in a vertex of $W \cup K$ and ends in a vertex of $Z$, which means that it must contain a vertex from $K=\{s_1,s_2\}$. Hence, $s_1 \in V(P[y,b]) \subseteq V(P)$. This proves the claim.
\end{proof}
Since $s_1$ is contained in none of the two non-empty sets $W' \cap Z$ and $Z' \cap Z$, Claim 9 shows that $D-s_1$ is not strongly connected, contradicting Claim 1. This shows that our very first assumption, namely that a digraph $D$ with $\vec{\chi}(D) \ge 4$ which does not contain a subdivision of $W_4^+$, exists, was wrong. This completes the proof of $\text{mader}_{\vec{\chi}}(W_4^+)=4$.
\end{proof}

\section{Mader-Perfect Digraphs and Open Problems}\label{sec:complete}

In this paper we investigated when the simple inequality
$\text{mader}_{\vec{\chi}}(F) \geq v(F)$ is tight. 
Observe that tightness is trivially preserved under taking spanning subdigraphs. 
It turns out however, that the optimality of the bound does not necessarily carry over to arbitrary subdigraphs. In fact, in Proposition~\ref{prop:addingisolanies} below we show that for any digraph $F$ there exists a constant $k_F$ such that adding $k_F$ isolated vertices to $F$ produces a digraph whose Mader number equals its number of vertices. 
This suggests that the class of digraphs $F$ which satisfy $\text{mader}_{\vec{\chi}}(F)=v(F)$ may not have a meaningful characterization.
This motivates the following definition. We call a digraph $F$ {\em Mader-perfect} if for every (induced) subdigraph $F'$ of $F$, the Mader number of $F'$ equals its order.

\begin{proposition}\label{prop:addingisolanies}
	For every digraph $F$ there exists $k_F \in \mathbb{N}$ such that for every $k \geq k_F$, the digraph $F'$ obtained from $F$ by adding $k$ new isolated vertices satisfies $\text{mader}_{\vec{\chi}}(F')=v(F')$. In fact, it suffices to take $k_F=2\cdot\text{mader}_{\vec{\chi}}(F)-v(F)-1$.
\end{proposition}
\begin{proof}
	Let $k \ge k_F:=2\cdot\text{mader}_{\vec{\chi}}(F)-v(F)-1$ be arbitrary. 
	Consider any given digraph $D$ such that $\vec{\chi}(D) \ge k+v(F)$. We need to show that $D$ contains a subdivision of $F$ which misses at least $k$ of the vertices of $D$. 
	
	Let $X \subseteq V(D)$ be a vertex set such that $\vec{\chi}(D[X]) = \text{mader}_{\vec{\chi}}(F)$. Then $D[X]$ contains a subdivision of $F$, and we have 
	$m := \vec{\chi}(D-X) \ge k+v(F)-\text{mader}_{\vec{\chi}}(F)$, for otherwise we could color $D$ with less than $k+v(F)$ colors. Let $Y_1,\ldots,Y_m$ be a partition of $V(D)\setminus X$ into $m$ acyclic sets. Let us first consider the case that at most $v(F)-1$ of these sets are singletons. Then 
	$$v(D)-|X|=|Y_1|+\ldots+|Y_m| \ge 2m-(v(F)-1)$$ $$\ge 2k+2v(F)-2 \cdot \text{mader}_{\vec{\chi}}(F)-(v(F)-1)= k+(k-k_F) \ge k.$$ Evidently, the subdivision of $F$ contained in $D[X]$ does not use any of the $\ge k$ vertices in $V(D)\setminus X$, concluding the proof in this case.
	
	In the other case, at least $v(F)$ of the sets $Y_i$ are singletons; without loss of generality, say $Y_i=\{y_i\}$ for $1 \le i \le v(F)$. Since $Y_1\ldots,Y_m$ form an optimal acyclic coloring of $D-X$, we cannot merge any two color classes to obtain an acyclic coloring with $m-1$ colors. It follows that $(y_i,y_j),(y_j,y_i) \in A(D)$ for every $1 \le i<j \le v(F)$. This implies that $D$ contains a copy of $F$ on the vertices $y_1,\ldots,y_{v(F)}$. By deleting the $v(F)$ vertices $y_1,\ldots,y_{v(F)}$, we obtain a digraph of dichromatic number at least $\vec{\chi}(D)-v(F) \ge k$, and hence, the remaining digraph consists of at least $k$ vertices. So we  see that $D$ indeed contains a subdivision of $F$ missing at least $k$ vertices in this second case as well. This concludes the proof.
\end{proof}  


Our main results --- namely Theorems \ref{thm:main1} and \ref{thm:dirdirac} --- can be restated as saying that all octi digraphs and all tournaments of order $4$ are Mader-perfect. For octi digraphs this follows immediately from Theorem \ref{thm:main1}, since octi are closed under taking subdigraphs; and for $4$-vertex tournaments this follows from the fact that every non-spanning subdigraph $F'$ of a $4$-vertex tournament is either an oriented triangle or an oriented path, and for those $F'$ the equality $\text{mader}_{\vec{\chi}}(F')=v(F')$ follows from Theorem \ref{thm:main1}. In a similar vein, Proposition \ref{K_3-e} implies that $\bivec{K}_3-e$ is Mader-perfect. Let us now give the proof.  
\begin{proof}[Proof of Proposition \ref{K_3-e}]
	It is sufficient to show that every $3$-dicritical digraph $D$ contains a subdivision of $\bivec{K}_3-e$. So let $D$ be a $3$-dicritical digraph. 
	Then $\delta^+(D),\delta^-(D) \geq 2$ and $D$ is strongly-connected, as guaranteed by Lemmas \ref{mindegree} and \ref{strong connectivity}, respectively. By Theorem \ref{critical strongly connected} applied for $k=1$, there is $v \in V(D)$ such that $D - v$ is strongly connected. Since $D$ is $3$-dicritical, there exists an acyclic $2$-coloring $c : V(D) \setminus \{v\} \rightarrow \{1,2\}$. Evidently, $c$ cannot be extended to an acyclic $2$-coloring of $D$. This means that for each $i = 1,2$, $D$ contains a directed cycle $C_i$ which contains $v$, such that all vertices in $V(C_i) \setminus \{v\}$ are colored with color $i$ (under $c$). Note that $V(C_1) \cap V(C_2) = \{v\}$. Since $D - v$ is strongly-connected, there is a path in $D$ from $V(C_1)$ to $V(C_2)$ which avoids $v$. Let $P$ be a shortest path from $V(C_1)$ to $V(C_2)$ avoiding $v$, and let us denote the endpoints of $P$ by $x_1,x_2$ (where $x_i \in V(C_i)$). The minimality of $P$ implies that $V(P) \cap V(C_i) = \{x_i\}$ (for each $i = 1,2$), since otherwise $P$ could be replaced by a shorter path. Now it is easy to see that the vertices $v,x_1,x_2$ and the (internally vertex-disjoint) dipaths $C_1[v,x_1], C_1[x_1,v],C_2[v,x_2],C_2[x_2,v],P$ form a subdivision of $\bivec{K}_3 - e$, as required.
\end{proof}

Altogether, we see that the class of Mader-perfect digraphs is quite rich. We believe it would be interesting to obtain a precise characterization of this class. 
\begin{problem}
	Characterize Mader-perfect digraphs. 
\end{problem}

On the negative side, $\bivec{K}_3$ is the smallest digraph $F$ satisfying $\text{mader}_{\vec{\chi}}(F) > v(F)$, hence no bioriented clique of order at least $3$ is Mader-perfect. 
In fact for any $k \ge 3$, the digraph obtained from $\bivec{K}_{k+2}$
by removing a bioriented $\bivec{C}_5$ has dichromatic
number $k$ but contains no subdivision of $\bivec{K}_k$. This shows that the Mader number of $\bivec{K}_k$ is at least $k+1$.

Already determining $\text{mader}_{\vec{\chi}}(\bivec{K}_3)$ exactly seems to be a challenging problem. From above we can only show that $\text{mader}_{\vec{\chi}}(\bivec{K}_3) \le 9$, where the upper bound follows from a combination of Proposition~\ref{K_3-e}  and Theorem \ref{addingarcs} below. We believe that the truth lies with the lower bound, provided by the above construction.
\begin{conjecture}
We have $\text{mader}_{\vec{\chi}}(\bivec{K}_3)=4$, i.e., every digraph $D$ with no $\bivec{K}_3$-subdivision admits an acyclic $3$-coloring.
\end{conjecture}

It is natural to ask how dense Mader-perfect digraphs can be. For $k \in \mathbb{N}$, let $m(k)$ denote the maximum possible number of arcs of a Mader-perfect digraph of order $k$. 
Using a variant of the classical probabilistic argument of Erd\H{o}s and Fajtlowicz \cite{erdosfajtlowicz}, we can show that $m(k)=O(k^{3/2}\sqrt{\log k})$, which means that Mader-perfect digraphs have to be (at least somewhat) sparse. 
In fact, let us show the slightly more general claim that 
\begin{equation}\label{eq:prob_lower_bound}
\text{mader}_{\vec{\chi}}(F) \ge \frac{cm^2}{k^2\log m}
\end{equation}
for every digraph $F$ on $k$ vertices and $m \geq c_1 k \log k$ arcs, where $c_1$ is a suitably large absolute constant. The bound \eqref{eq:prob_lower_bound} shows that if $\text{mader}_{\vec{\chi}}(F) = v(F) = k$ (which has to be the case if $F$ is Mader-perfect), then $m = O(k^{3/2}\sqrt{\log k})$, as claimed.

To prove \eqref{eq:prob_lower_bound}, consider any fixed digraph $F$ consisting of $k \ge 2$ vertices and 
$m \geq c_1 k \log k$ arcs. Let $D(n,p)$ be the random digraph\footnote{Recall that $D(n,p)$ is the random digraph on the vertex-set $\{1,\ldots,n\}$, where for each $1 \leq i \neq j \leq n$ we put the arc $(i,j)$ independently with probability $p$.} with parameters $n=\lfloor \frac{m}{2}\rfloor$ and $p=\frac{m}{4k^2}$. We claim that with positive probability, $D(n,p)$ contains neither a set of $k$ vertices spanning at least $m/2$ arcs nor an acyclic set of more than $\frac{c_2k^2\log m}{m}$ vertices for some suitable absolute constant $c_2$. To see this, note that the expected number of arcs spanned by some fixed set of $k$ vertices in $D(n,p)$ is $pk(k-1)<\frac{m}{4}$, and hence the Chernoff-bound yields that the probability that some $k$ fixed vertices span at least $\frac{m}{2}$ arcs is bounded by $\exp(-\frac{m}{12})$. Therefore the probability that there are $k$ vertices spanning at least $\frac{m}{2}$ arcs is at most 
$\binom{n}{k}\exp(-\frac{m}{12}) \leq (m/2)^k\exp(-\frac{m}{12}) = 
\exp(k\log (m/2)-\frac{m}{12})<\frac{1}{2}$, provided $c_1$ is chosen large enough. 

Similarly, the probability that any fixed set of $\alpha > \frac{c_2k^2\log m}{m}$ vertices is acyclic in $D(n,p)$ is at most $\alpha!(1-p)^{\binom{\alpha}{2}}$. Hence, the probability that $D(n,p)$ contains an acyclic set of size at least $\alpha$ is at most $\binom{n}{\alpha}\alpha!(1-p)^{\binom{\alpha}{2}} \le \exp(\alpha \log n-p\binom{\alpha}{2})<1/2$, provided $c_2$ is chosen large enough (where in the last inequality we plugged in our choice of $n$, $p$ and $\alpha$).

We conclude that there exists a digraph $D$ on $n=\lfloor \frac{m}{2}\rfloor$ vertices containing no $k$ vertices spanning at least $\frac{m}{2}$ arcs and whose dichromatic number is at least $$\vec{\chi}(D) \ge \frac{n}{\left(\frac{c_2k^2\log m}{m}\right)}=\Omega\left(\frac{m^2}{k^2\log m}\right).$$  
Observe that $D$ contains no subdivision of $F$; indeed, if $D$ contained a subdivision of $F$, then since $D$ has at most $\frac{m}{2}$ vertices, at least $\frac{m}{2}$ of the subdivision paths would have to be of length $1$, i.e. be ``direct" arcs between the $k$ branch vertices of the subdivision. But this is impossible as $D$ contains no set of $k$ vertices spanning at least $\frac{m}{2}$ arcs, a contradiction. 
This proves \eqref{eq:prob_lower_bound}. 

We note that if $F$ is symmetric, i.e. if it is a biorientation of an undirected graph, then we can improve the bound \eqref{eq:prob_lower_bound} to $\text{mader}_{\vec{\chi}}(F) \ge \Omega(m/\log m)$. To see this, let $D$ be a tournament of order $m/2$ and dichromatic number $\Omega(m/\log m)$ (it is well-known that such tournaments exist, see \cite{Erdos}). Then $D$ contains no subdivision of $F$. Indeed, since $D$ contains no digons, any subdivision of $F$ in $D$ must contain at least $m/2$ subdivision vertices, one per every digon in $F$. But as $v(D) = m/2 < m/2 + v(F)$, there are not enough vertices in $D$ to fit a subdivision of $F$. As a corollary, we see that if $F$ is Mader-perfect and symmetric, then $a(F) \leq O(k \log k)$, \nolinebreak where \nolinebreak $k = \nolinebreak v(F)$. 


So far we have shown that $m(k)=O(k^{3/2}\sqrt{\log k})$.
As for a lower bound, consider the digraph obtained from $\bivec{K}_3 - e$ by performing $k-3$ ear additions, where at each step we attach a digon to the existing digraph. Then the resulting digraph $F_k$ has $k$ vertices and $2k-1$ arcs. By combining Proposition \ref{K_3-e} with Theorem \ref{kempe}, we see that 
$F_k$ is Mader-perfect. Hence, $m(k) \geq 2k-1$. 
It would be interesting to close the gap between the upper and lower bounds. We conjecture that the truth lies with the latter. 
\begin{conjecture}
	$m(k)=O(k)$.
\end{conjecture}

Aboulker et al.~\cite{aboulker} studied the behaviour of the Mader number with respect to the insertion of arcs, and proved the following bound using a beautiful argument based on breadth-first-search \nolinebreak trees.
\begin{theorem}[\cite{aboulker}, Lemma 31]\label{addingarcs}
If $F$ is a digraph and $e \in A(F)$, then
$$\text{mader}_{\vec{\chi}}(F) \le 4\cdot \text{mader}_{\vec{\chi}}(F-e)-3.$$ 
\end{theorem}
Using this upper bound and a lower-bound-construction for tournaments, they obtained the following bounds on $\text{mader}_{\vec{\chi}}(\bivec{K}_n)$.
\begin{theorem}
For every $n \in \mathbb{N}$, 
$$\Omega\left(\frac{n^2}{\log n}\right) \le \text{mader}_{\vec{\chi}}(\bivec{K}_n) \le 4^{n^2-2n+1}(n-1)+1.$$
\end{theorem}
Consider the digraph $F_n$ on $n$ vertices and $2n-1$ as constructed above. Then we have
$\text{mader}_{\vec{\chi}}(F_n) = n$, and hence by starting from $F_n$ and repeatedly applying Theorem \ref{addingarcs}, we get the (slightly) improved bound 
$\text{mader}_{\vec{\chi}}(\bivec{K}_n) \leq 4^{n^2-3n+1}(n-1) + 1$.   

Still, the gap between the lower and upper bounds on $\text{mader}_{\vec{\chi}}(\bivec{K}_n)$ remains huge. Unfortunately, the techniques used in this paper to tackle sparse digraphs do not seem to allow for substantial improvements of the upper bound. 
Our attempts to improve the lower bound to a super-quadratic growth have also been unsuccessful. 
It is tempting to conjecture the following:
\begin{conjecture}\label{con:quadratic}
There exists an absolute constant $c>0$ such that $\text{mader}_{\vec{\chi}}(\bivec{K}_n) \le cn^2$ for every positive integer $n$.
\end{conjecture}
It is worth noting that by a result of Gir\~{a}o et al.~\cite{GPS}, every tournament $T$ with minimum out-degree at least $cn^2$ (for some absolute constant $c$) contains a subdivision of $\bivec{K}_n$. 
This implies that {\em in tournaments}, having dichromatic number larger than $cn^2$ forces a subdivision of $\bivec{K}_n$. 
As mentioned in the introduction, however, extending the result of \cite{GPS} to general digraphs is impossible, since having large minimum out- and in-degree does not force $\bivec{K}_n$-subdivisions in general digraphs for any $n \geq 3$.  


Another intriguing question is to determine the Mader number of bioriented cycles. Specifically, is it the case that $\text{mader}_{\vec{\chi}}(\bivec{C_{\ell}}) = \ell$ for all $\ell \geq 4$?
\begin{problem}
	What is $\text{mader}_{\vec{\chi}}(\bivec{C_{\ell}})$?
\end{problem}

A related problem is to determine the maximum possible chromatic number of a digraph which does not contain a subdivision of {\em any} bioriented cycle. We conjecture that the answer is \nolinebreak $2$.
\begin{conjecture}\label{conj:bicycles}
	Let $D$ be a digraph with $\vec{\chi}(D) \geq 3$. Then there is $\ell \geq 3$ such that $D$ contains a subdivision of $\bivec{C_{\ell}}$. 
\end{conjecture}


{\bf Acknowledgement} The research on this project was initiated during a joint research workshop of Tel Aviv University and the Freie Universit\"at Berlin on Ramsey Theory, held in Tel Aviv in March 2020, and partially supported by GIF grant G-1347-304.6/2016. We would like to thank the German-Israeli Foundation (GIF) and both institutions for their support.


\begin{thebibliography}{99}

    \bibitem{aboulker}
    P. Aboulker, N. Cohen, F. Havet, W. Lochet, P. S. Moura and S. Thomass\'{e}, Subdivisions in digraphs of large out-degree or large dichromatic number. \emph{Electronic Journal of Combinatorics}, 26(3), P3.19, 2019.
	\bibitem{perfect}
	S. D. Andres and W. Hochst\"{a}ttler. Perfect digraphs. \emph{Journal of Graph Theory}, 79, 21-29, 2015.
	\bibitem{BJ-G}
	J. Bang-Jensen and G. Z. Gutin, Digraphs: theory, algorithms and applications. \emph{Springer Science \& Business Media}, 2008. 
	\bibitem{lists}
	J. Bensmail, A. Harutyunyan and N. Le. List coloring digraphs. \emph{Journal of Graph Theory}, 87, 492-508, 2018.
	\bibitem{bollobas}
    B. Bollob\'{a}s and A. Thomason. Proof of a conjecture of Mader, Erd\H{o}s and Hajnal on topological complete subgraphs. \emph{European Journal of Combinatorics}, 19(8), 883-887, 1998.
    \bibitem{burr}
	S. Burr. Subtrees of directed graphs and hypergraphs. In \emph{Proceedings of the Eleventh Southeastern Conference on Combinatorics, Graph Theory and Combinatorics (Florida Atlantic Univ., Boca Raton, Fla., 1980, Vol. I}, 28, 227-239, 1980.
	\bibitem{catlin}
	P. A. Catlin. Haj\'{o}s' graph coloring conjecture: variations and counterexamples. \emph{Journal of Combinatorial Theory, Series B}, 26(2), 268-274, 1979.
	\bibitem{chromorientedcycles}
	N. Cohen, F. Havet, W. Lochet and N. Nisse. Subdivisions of oriented cycles in
digraphs with large chromatic number. \emph{Journal of Graph Theory}, 89(4), 439-456, 2018.
	\bibitem{dirac4}
	G. A. Dirac. A property of $4$-chromatic graphs and some remarks on critical graphs. \emph{J. London Math. Soc.}, 27, 85-92, 1952.
	\bibitem{Erdos}
	Paul Erd\H{o}s, Problems and results in number theory and graph theory, Proc.
	Ninth Manitoba Conf. on Numerical Math. and Computing 1979, pp. 3–21.
	\bibitem{erdosfajtlowicz}
	P. Erd\H{o}s and S. Fajtlowicz. On the conjecture of Haj\'{o}s. \emph{Combinatorica}, 1(2), 1981, 141-143.
	\bibitem{foxetal}
	J. Fox, C. Lee and B. Sudakov. Chromatic number, clique subdivisions, and the conjectures of Haj\'{o}s and Erd\H{o}s-Fajtlowicz. \emph{Combinatorica}, 33, 181-197, 2013. 

	\bibitem{GPS}
	A. Gir\~{a}o, K. Popielarz and R. Snyder, Subdivisions of digraphs in tournaments. \emph{arXiv preprint}, arXiv:1908.03733, 2019. 
      \bibitem{Goring}
        F. G\"oring, Short proof of Menger's Theorem. \emph{Discrete Mathematics}, 219, 295-296, 2000.
      \bibitem{hadwiger}
	H. Hadwiger. \"{U}ber eine Klassifikation der Streckenkomplexe. \emph{Vierteljschr. Naturforsch. Ges. Z\"{u}rich}, 88, 133-143, 1943.	
	\bibitem{hajosoriginal}
	G. Haj\'{o}s. \"{U}ber eine Konstruktion nicht $n$-f\"{a}rbbarer Graphen. \emph{Wiss. Z. Martin-Luther-Univ. Halle-Wittenberg. Math.-Nat. Reihe}, 10, 116-117, 1961.
	\bibitem{HARUTYUNYAN2019}
	A. Harutyunyan, T.-N. Le, S. Thomass\'{e} and H. Wu. Coloring tournaments: From local to global. \emph{Journal of Combinatorial Theory, Series B}, 138, 166-171, 2019.
	\bibitem{dig5}
	A. Harutyunyan and B. Mohar. Planar digraphs of digirth five are $2$-colorable. \emph{Journal of Graph Theory}, 84, 408-427, 2017.
	\bibitem{HMM}
	F. Havet, A. K. Maia and B. Mohar. 
	Finding a subdivision of a prescribed digraph of order 4. 
	\emph{Journal of Graph Theory}, 87(4), 536-560, 2018. 
	
	\bibitem{Komlos}
	J. Koml\'{o}s and E. Szemer\'{e}di (1996). Topological cliques in graphs II.\emph{Combinatorics, Probability and Computing}, 5(1), 79-90. 
	\bibitem{dig4}
	Z. Li and B. Mohar. Planar digraphs of digirth four are $2$-colorable. \emph{SIAM Journal on Discrete Mathematics}, 31(3), 2201-2205.
	\bibitem{Mader_critical_connectivity}
	W. Mader, Ecken von kleinem Grad in kritisch $n$-fach zusammenh\"{a}ngenden Digraphen, \emph{Journal of Combinatorial Theory, Series B}, 53(2), 260-272, 1991 (in German). 
	\bibitem{Mader_trans_4}
	W. Mader, On topological tournaments of order 4 in digraphs of outdegree 3. \emph{Journal of Graph Theory}, 21, 371-376,  1996. 
	\bibitem{mengeroriginal}
	K. Menger. Zur allgemeinen Kurventheorie. \emph{Fundamenta Mathematicae}, 10(1), 96-115, 1927.
	\bibitem{fractionalNL}
	B. Mohar and H. Wu. Dichromatic number and fractional chromatic number. \emph{Forum of Mathematics, Sigma}, 4, E32, 2016.
	\bibitem{MoharEdgeWeight}
	B. Mohar. Circular colorings of edge-weighted graphs. \emph{Journal of Graph Theory}, 43, 107-116, 2003.
	\bibitem{neulara}
	V. Neumann-Lara. The dichromatic number of a digraph. \emph{Journal of Combinatorial Theory, Series B}, 33(3), 265-270, 1982.
	\bibitem{thom85}
	C. Thomassen. Even cycles in directed graphs. \emph{European Journal of Combinatorics},
6(1), 85-89, 1985.
	\bibitem{thom05}
	C. Thomassen. Some remarks on Haj\'{o}s' conjecture. \emph{Journal of Combinatorial Theory, Series B}, 93(1), 95-105, 2005.
\end{thebibliography}
\end{document}